\documentclass[12pt]{extarticle}
\usepackage[utf8]{inputenc}
\usepackage[english]{babel}

\usepackage{caption}
\oddsidemargin \evensidemargin
\usepackage{amsthm}
\usepackage{mathtools}
\usepackage{amssymb}
\usepackage[section]{algorithm}
\usepackage{algorithmic}
\usepackage{blkarray}
\usepackage{url}
\usepackage{tikz}
\usepackage{pgfplots}
\usepackage{slashbox}
\usepackage{booktabs}
\def\rotbra#1{\rotatebox{90}{$#1$}}
\usepackage[font=footnotesize,labelfont=bf]{caption}
 
\usetikzlibrary{external}
\usetikzlibrary{patterns}
\pgfplotsset{
  log x ticks with fixed point/.style={
      xticklabel={
        \pgfkeys{/pgf/fpu=true}
        \pgfmathparse{exp(\tick)}%
        \pgfmathprintnumber[fixed relative, precision=3]{\pgfmathresult}
        \pgfkeys{/pgf/fpu=false}
      }
  }}
\usepackage{array}
\usepackage{hyperref}
\hypersetup{
    colorlinks=true,
    linkcolor=orange!80!black,
    filecolor=magenta,      
    urlcolor=cyan, 
    citecolor=blue!50!white,
  }
\usepackage{subfig}
\usepackage{todonotes}
\usepackage{cleveref}

\newcommand{\algoname}{\texttt{cpd\_hnf}}
\providecommand{\keywords}[1]{\small \textbf{Key words ---} #1}

\title{A normal form algorithm for tensor rank decomposition}
 \author{Simon Telen and Nick Vannieuwenhoven}
 \date{}

\theoremstyle{plain}
\newtheorem{theorem}{Theorem}[section]
\newtheorem{conjecture}{Conjecture}
\newtheorem{lemma}{Lemma}[section]

\newtheorem{remark}{Remark}[section]
\newtheorem{proposition}{Proposition}[section]
\newtheorem{corollary}{Corollary}[section]

\theoremstyle{definition}
\numberwithin{equation}{section}
\newtheorem{definition}{Definition}[section]

\newtheorem{examplex}{Example}[section]
\newenvironment{example}
  {\pushQED{\qed}\examplex}
  {\popQED\endexamplex}
\newtheorem{assumption}{Assumption}[section]

\newcommand{\N}{\mathbb{N}}

\newcommand{\HF}{\textup{HF}}
\newcommand{\Reg}{\textup{Reg}}
\newcommand{\id}{\textup{id}}

\DeclareMathOperator{\im}{\textup{im}}

\newcommand{\B}{\mathcal{B}}

\newcommand{\diag}{\textup{diag}}

\newcommand{\rank}{\textup{rank}}

\newcommand{\A}{\mathcal{A}}

\newcommand{\M}{\mathcal{M}}

\newcommand{\Rc}{\mathcal{R}}
\newcommand{\C}{\mathbb{C}}

\newcommand{\PP}{\mathbb{P}}

\newcommand{\1}{\mathbf{1}}
\newcommand{\z}{\zeta}

\renewcommand{\d}{\partial}

\renewcommand{\z}{\zeta}

\newcommand{\R}{\mathbb{R}}

\newcommand{\Span}{\textup{span}}
\newcommand{\sumab}{\sum_{a,b}}
\newcommand{\sumapbp}{\sum_{a',b'}}
\newcommand{\sumkl}{\sum_{k,l}}

\providecommand{\keywords}[1]{\small \textbf{Key words ---} #1}

\newcommand{\ideal}[1]{\left \langle {#1}  \right \rangle}

\newcommand{\blue}[1]{\textcolor{blue!50!white}{#1}}
\newcommand{\red}[1]{\textcolor{red!50!black}{#1}}
\newcommand{\green}[1]{\textcolor{green!50!black}{#1}}
\newcommand{\orange}[1]{\textcolor{orange!80!black}{#1}}
\newcommand{\brown}[1]{\textcolor{blue!70!green}{#1}}
\newcommand{\magenta}[1]{\textcolor{magenta!80!black}{#1}}

\newcommand*{\matminus}{%
  \leavevmode
  \hphantom{0}%
  \llap{%
    \settowidth{\dimen0 }{$0$}%
    \resizebox{1.1\dimen0 }{\height}{$-$}%
  }%
}

\begin{document}

\maketitle

\begin{abstract}
We propose a new numerical algorithm for computing the tensor rank decomposition or canonical polyadic decomposition of higher-order tensors subject to a rank and genericity constraint. Reformulating this computational problem as a system of polynomial equations allows us to leverage recent numerical linear algebra tools from computational algebraic geometry. We characterize the complexity of our algorithm in terms of an algebraic property of this polynomial system---the multigraded regularity. 
We prove effective bounds for many tensor formats and ranks, which are of independent interest for overconstrained polynomial system solving. Moreover, we conjecture a general formula for the multigraded regularity, yielding a (parameterized) polynomial time complexity for the tensor rank decomposition problem in the considered setting. Our numerical experiments show that our algorithm can outperform state-of-the-art numerical algorithms by an order of magnitude in terms of accuracy, computation time, and memory consumption.
\end{abstract}


\maketitle

\section{Introduction} \label{sec:intro}

We introduce an original \textit{direct} numerical algorithm for \textit{tensor rank decomposition} or \textit{canonical polyadic decomposition} (CPD) in the low-rank regime. By ``direct'' we mean an algorithm that does not rely on numerical optimization or other iteratively refined approximations with a data-dependent number of iterations.

Consider the vector space $ \C^{(n_1 + 1) \times \cdots \times (n_{d} + 1)}$ whose elements represent order-$d$ tensors in coordinates relative to a basis. We say that a tensor in such a space is \emph{of rank 1} or \emph{elementary} if it is of the following form:
\[
{\alpha}^1 \otimes {\alpha}^2 \otimes \cdots \otimes {\alpha}^d := \left ( \alpha^1_{j_1} \alpha^2_{j_2} \cdots \alpha^d_{j_d} \right )_{\substack{0 \le j_k \le n_k \\ k=1,\ldots,d}}, \quad \text{where} \quad {\alpha}^k = \left(\alpha^k_{j} \right)_{0 \le j \le n_k} \in \C^{n_k + 1}.
\]
Every tensor $\mathcal{A}$ can be expressed as a linear combination of rank-$1$ tensors:
\begin{align}\tag{CPD}\label{eqn_cpd}
 \mathcal{A} = \sum_{i=1}^r {\alpha}_i^1 \otimes \cdots \otimes {\alpha}_i^d, \quad \text{with} \quad \alpha^k_i = \left(\alpha^k_{i,j} \right)_{0 \le j \le n_k} \in \C^{n_k + 1}.
\end{align}
If $r$ is minimal among all such expressions of $\mathcal{A}$, then $r$ is called the \textit{rank} of the tensor according to \cite{Hitchcock1927}, and \cref{eqn_cpd} is called a CPD. 
The problem of computing a CPD of $\A$, i.e., determining a set of rank-$1$ tensors summing to $\A$, has many applications in science and engineering, see for instance \cite{kolda2009tensor} and \cite{sidiropoulos2017tensor}. 

The strategy we propose for computing \eqref{eqn_cpd} relies on the fact that the problem is equivalent to solving a certain system of polynomial equations. Under suitable assumptions, these equations can be obtained from the nullspace of a \emph{flattening} $\A_{(1)} \in \C^{(n_1+1) \times \prod_{k=2}^d (n_k+1)}$ of the tensor $\A$, as in \cite{kuo2018computing}. Once we have obtained these polynomial equations, whose solutions correspond to the rank-1 terms in \eqref{eqn_cpd}, we use recent numerical normal form techniques from  \cite{bender2020toric} and \cite{telen2020thesis,telen2019numerical} to solve them. The following example, which is used as a running example throughout the paper, illustrates how this works. 

\begin{example}[Motivating example] \label{ex:running1}
Consider the $4 \times 3 \times 3$ tensor $\A$ with flattening
 \[ \A_{(1)} = \begin{blockarray}{ccccccccc}
\d_{00} & \d_{01} & \d_{02} &  \d_{10} & \d_{11} & \d_{12}& \d_{20} & \d_{21} & \d_{22} \\
\begin{block}{[ccccccccc]}
1 & 0 & 0 & 0 & 0 & 0 & 2 & 0 & 0\\
1 & 1 & 0 & 0 & 0 &0 & 2 & 1 & 0 \\
1 & 1 & 1 & 0 & 0 & 1 & 2 & 1 & 2 \\
1 & 1 & 1 & 1 & 1 & 2 & 2 & 1 & 2 \\
\end{block}
\end{blockarray}.
\]
The column of this matrix indexed by $\d_{kl}$ contains the entries $\A_{jkl},~ j = 0, \ldots, 3$. The reason for this indexing will become clear in \cref{sec:tensortopol}. The kernel of $\A_{(1)}$ is the transpose of 
 \begin{equation} \label{eq:exmacmtx}
 \makeatletter\setlength\BA@colsep{3.3pt}\makeatother
 R_I{(1,1)}^\top = \begin{blockarray}{cccccccccc}
&x_0y_0 & x_0y_1 & x_0y_2 & x_1y_0 & x_1y_1 & x_1y_2 & x_2y_0 & x_2y_1 & x_2y_2 \\
\begin{block}{c[ccccccccc]}
f_1 & 0 & 0 & 0 & -1 & 1 & 0 & 0 & 0 & 0 \\
f_2 & 0 & 0 & -1 & -1 & 0 & 1 & 0 & 0 & 0 \\
f_3 & -2 & 0 & 0 & 0 & 0 & 0 & 1 & 0 & 0 \\
f_4 & 0 & -1 & 0 & 0 & 0 & 0 & 0 & 1 & 0 \\
f_5 & 0 & 0 & -2 & 0 & 0 & 0 & 0 & 0 & 1\\
\end{block}
\end{blockarray}~.
\end{equation}
The column of $R_I{(1,1)}^\top$ corresponding to column $\d_{kl}$ of $\A_{(1)}$ is now indexed by $x_ky_l$: we interpret the rows as polynomials 
\begin{align*}
f_1 &= -x_1y_0 + x_1y_1, \quad f_2 = -x_0y_2 - x_1y_0 + x_1y_2, \quad f_3 = -2x_0y_0 + x_2y_0, \\
f_4 &= -x_0y_1 + x_2y_1, \quad f_5 = -2x_0y_2 + x_2y_2.
\end{align*}
These are bilinear forms in $S = \C[x_0,x_1,x_2,y_0,y_1,y_2]$. As explained in \cref{sec:tensortopol}, the common zeros of $f_1, \ldots, f_5$ form a subvariety of $\PP^2 \times \PP^2$ consisting of the four points 
\[ \begin{matrix}
 \z_1 = ((1:0:2),(1:0:0)), &  \z_2 = ((1:0:1),(0:1:0)),\\ \z_3 = ((1:1:2) ,(0:0:1)),& \z_4 = ((0:1:0),(1:1:1)).
\end{matrix}
\]
These points correspond to the last two factors of the rank-1 terms in 
$$
\begin{bmatrix}
1\\1\\1\\1
\end{bmatrix} \otimes \underbrace{\begin{bmatrix}
1\\0\\2
\end{bmatrix}
\otimes
\begin{bmatrix}
1\\0\\0
\end{bmatrix}}_{\z_1}
~+~
\begin{bmatrix}
0\\1\\1\\1
\end{bmatrix}
\otimes
\underbrace{\begin{bmatrix}
1\\0\\1
\end{bmatrix}
\otimes
\begin{bmatrix}
0\\1\\0
\end{bmatrix}}_{\z_2}
~+~
\begin{bmatrix}
0\\0\\1\\1
\end{bmatrix}
\otimes
\underbrace{\begin{bmatrix}
1\\1\\2
\end{bmatrix}
\otimes
\begin{bmatrix}
0\\0\\1
\end{bmatrix}}_{\z_3}
~+~
\begin{bmatrix}
0\\0\\0\\1
\end{bmatrix}
\otimes 
\underbrace{\begin{bmatrix}
0\\1\\0
\end{bmatrix}
\otimes
\begin{bmatrix}
1\\1\\1
\end{bmatrix}}_{\z_4}
;
$$
this is the decomposition \eqref{eqn_cpd} of $\A$. 
\end{example}

\subsection*{Contributions.}
We formulate our results for the case $d = 3$, as the general case can be handled using the standard \textit{reshaping trick} (\cref{subsec:reshapingtrick}). 
Our main contribution is a new numerical algorithm for computing the rank-$r$ CPD of a third-order tensor $\A$ based on linear algebra, under the assumption that $\A$ is \emph{$r$-identifiable} and a \emph{flattening of $\A$ has rank $r$}, see \cref{assum:vanishingideal}. We call this algorithm \algoname{}. It is based on a well-known reformulation of the problem as a system of polynomial equations, which we solve using state-of-the-art methods from computational algebraic geometry. This results in \cref{alg:mhnfalg}. We show that this algorithm generalizes \emph{pencil-based} algorithms \cite{BBV2019,LRA1993,Lorber1985,SK1990,Lathauwer2006} from very low ranks to much higher ranks in the unbalanced case; see \cref{thm:MHNF,thm:MHNFlowrank,sec_sub_pba}. 

We give a new, explicit description of the complexity of the tensor rank decomposition problem in terms of an algebraic property of aforementioned polynomial system: the \emph{multigraded regularity} of a non-saturated ideal in the homogeneous coordinate ring of $\PP^m \times \PP^n$ (\cref{prop:checkhf}). We characterize the regularity in terms of the rank of a structured matrix obtained directly from the rank-$1$ terms in \eqref{eqn_cpd}, see \cref{ex:reg1}. These new insights allow us to formulate a conjecture regarding the regularity (\cref{conj:reg}). We prove this conjecture for many formats, see \cref{thm:conjcases}. These results are of independent interest for the field of polynomial system solving. They have the following consequence related to the complexity of our algorithm. 
\begin{theorem} \label{prop_polynomial}
 Consider the tensor space $\C^{\ell+1}\otimes\C^{m+1}\otimes\C^{n+1}$ of dimension $M = (\ell+1)(m+1)(n+1)$ with $\ell\ge m \ge n$. If \cref{conj:reg} holds and $\A$ is a generic\footnote{A property on a variety $\mathcal{V}$ is ``generic'' if the locus where the property does not hold is contained in a Zariski closed subset.} tensor of rank $r \le \phi m n$ with $\phi \in [0,1)$ a fixed constant, then \textup{\algoname{}} runs in polynomial time $\mathcal{O}( M^{\frac{5}{2} \lceil \frac{1}{1-\phi} \rceil +1} )$.
\end{theorem}
Our numerical experiments in \cref{sec:experiments} show that the proposed algorithm is highly efficient and gives accurate results (see \cref{fig_accuracy}). For instance, we can compute the decomposition of a $7 \times 7 \times 7 \times 7 \times 6 \times 6 \times 5 \times 5$ tensor of rank 1000 in double-precision arithmetic with an accuracy of order $10^{-15}$ in 441 seconds---a feat we believe has not been matched by other tensor decomposition algorithms. Moreover, \algoname{} seems to behave well in the presence of noise. 

A Julia implementation of \algoname{}, including driver routines to reproduce our numerical experiments, is provided at \url{https://gitlab.kuleuven.be/u0072863/homogeneous-normal-form-cpd}.

\subsection*{Related work.}
The idea of computing tensor decompositions via polynomial root finding is central in \emph{apolarity}-based approaches such as \cite{bernardi2011multihomogeneous,bernardi2013general, bernardi2020waring}, and \cite{brachat2010symmetric,nie2017generating} for the symmetric case.
The \emph{Hankel operators} play the role of normal forms in this context. These operators can be obtained partially from the tensor $\A$. In most cases, an additional polynomial system needs to be solved in order to complete the Hankel operators \cite[Section 4]{mourrain2018polynomial}. For instance, this is step (2) in \cite[Algorithm 5.1]{brachat2010symmetric}. Although our method works only under certain assumptions on the rank of $\A$, in contrast to apolarity-based methods it requires \emph{only linear algebra computations}, and it operates in polynomial rings with fewer variables. Moreover, the method from \cite{brachat2010symmetric} uses \emph{connected-to-one bases} for its normal form computations. The choice of such a basis is discussed at length in \cite[Sections 4.3 and 7.2]{bernardi2020waring}. In this paper, we exploit the flexibility of \emph{truncated normal forms} \cite{telen2018solving} to achieve better numerical results. For a comparison, see \cite[Section 4.3.3]{telen2020thesis}.

\cite{kuo2018computing} obtained an affine system of polynomial equations as in \cref{ex:running1}, but it is solved using homotopy continuation methods. This approach is infeasible for some formats that are handled without any problems by our algorithm. For instance, for the aforementioned eighth-order tensor of rank 1000, the method of \cite{kuo2018computing} would need to track over $40 \cdot 10^9$ homotopy paths. We argue that the eigenvalue-based methods proposed in this paper are more natural to use in an overdetermined setting.

The state-of-the-art algorithms for tensor rank decomposition using only linear algebra computations were proposed by \cite{domanov2014canonical,domanov2017canonical}. Although these methods work under slightly milder conditions, our numerical experiments suggest that \cite[Algorithm 1]{domanov2017canonical} often requires the construction of larger matrices than those in our algorithm. There is no explicit connection with polynomial equations. The algorithm and its complexity depend on a parameter $l$, which is chosen incrementally by trial and error for each format. In an analogous way, the complexity of our algorithm is governed by the choice of a parameter. However, our analysis in \cref{sec:regorder3} tells us \emph{a priori} which parameter value should be used, circumventing a trial-and-error approach. In \cref{sec:experiments} we demonstrate that \algoname{} improves on \cite{domanov2014canonical,domanov2017canonical} in terms of computational complexity and accuracy. 

\subsection*{Outline.}
In \cref{sec:tensortopol}, we state our assumptions and show how computing the CPD of a tensor $\A$ is formulated as a system of polynomial equations. In \cref{sec:normalforms}, we make the connection with normal form methods explicit. That is, we describe how a \emph{pre-normal form} can be computed directly from the tensor $\A$ and how this allows us to reduce the above polynomial system to an eigenvalue problem. We explain how the approach generalizes so-called \emph{pencil-based algorithms} in \cref{sec_sub_pba}. This leads to a complete, high-level description of our algorithm \algoname{} in \cref{sec_sub_hnfalgorithm}. In \cref{sec:regorder3}, we study the regularity of the ideal associated to our polynomial system. These results are the starting point for our analysis of the complexity of \algoname{} in \cref{sec:complexity}. In \cref{sec:experiments}, we demonstrate the efficiency and accuracy of \algoname{} relative to the state of the art through several numerical experiments. The paper is completed in \cref{sec:conclusions} with some final conclusions.

\section{From tensor decomposition to polynomial equations} \label{sec:tensortopol}

In this section, we explain how tensor rank decomposition, under some restrictions, can be reduced to solving a polynomial system of equations whose coefficients are directly obtained from the tensor $\mathcal{A}$. The next two subsections state the restrictions under which the proposed algorithm operates. In \cref{subsec:polfromflat}, the polynomial system is constructed. We show that it gives a nice algebraic description of a certain projection of the rank-$1$ tensors appearing in $\mathcal{A}$'s decomposition. \Cref{sec_sub_highlevel} provides a pseudo-algorithm summarizing the main steps of the proposed numerical algorithm.

\subsection{Identifiability, flattenings, and the main assumption}

For the moment, let us assume $\mathcal{A} \in \C^{\ell+1}\otimes\C^{m+1}\otimes\C^{n+1}$ is a third-order tensor of rank $r$. The \algoname{} algorithm works under the following assumption.
\begin{assumption} \label{assum:vanishingideal}
The tensor $\A \in \C^{\ell+1} \otimes \C^{m+1} \otimes \C^{n+1}$ with $\ell \ge m \ge n > 0$ is generic of rank
\begin{align}\tag{R} \label{eqn_rank_condition}
 r \le \min\left\{ \ell+1, mn \right\}.
\end{align}
\end{assumption}
By \emph{generic} we mean that $\A$ is contained in a Zariski dense open subset of the set of all rank-$r$ tensors. We will describe this open subset more explicitly in \cref{lem:zerodim}. For now, we point out that in this open set, each tensor $\A$ is such that
\begin{enumerate}
\item[(i)] $\mathcal{A}$ is \textit{$r$-identifiable}, and 
\item[(ii)] the standard 1-\textit{flattening} $\A_{(1)}$ of the tensor $\A$ is of the same rank as $\A$.
\end{enumerate}
These are necessary conditions for our algorithm to work. We now briefly recall their meaning.

The condition (i) is usually very weak and only of a technical nature.
Recall that the set of all rank-$1$ tensors forms an \textit{algebraic variety}, i.e., the solution set of a system of polynomial equations, called the Segre variety $\mathcal{S}$. 
A rank-$r$ CPD of a tensor $\mathcal{A}$ is a set of $r$ rank-$1$ tensors whose sum is $\mathcal{A}$. The set of all such rank-$r$ CPDs is denoted by $\mathcal{S}^{[r]} = \{ \mathcal{X} \subset \mathcal{S} \mid |\mathcal{X}|=r \}$.
\textit{Tensor rank decomposition} consists of computing an element of the fiber of 
\[
 f: \mathcal{S}^{[r]} \to \C^{\ell+1} \otimes \C^{m+1} \otimes \C^{n+1},\quad \{\mathcal{A}_1, \ldots, \mathcal{A}_r\} \mapsto \mathcal{A}_1 + \cdots + \mathcal{A}_r
\]
at a rank-$r$ tensor $\mathcal{A}$. For brevity, we denote the image of $f$ by 
\[
\mathcal{S}_r = f(\mathcal{S}^{[r]}).
\]
The tensor $\mathcal{A}$ is called $r$-identifiable if the fiber $f^{-1}(\mathcal{A})$ contains exactly one element. That is, there is a unique set in $\mathcal{S}^{[r]}$ whose sum is $\mathcal{A}$. 
Generally, $r$-identifiability fails only on a strict closed subvariety of (the Zariski closure of) $\mathcal{S}_r$; see  \cite{CO2012,COV2014,COV2017,BCO2013}. This property is called the \textit{generic $r$-identifiability of $\mathcal{S}$}.
These results entail that $r$-identifiability fails only on a subset of $\mathcal{S}_r$ of Lebesgue measure zero if the rank $r$ and dimensions $(\ell+1,m+1,n+1)$ satisfy some very weak conditions; see \cite[Section 1]{COV2014} and \cite[Section 3]{COV2017} for more details and statements for higher-order tensors as well. 
If $\mathcal{S}_r$ is generically $r$-identifiable, then there is a Zariski-open subset of the closure of $\mathcal{S}_r$ so that $f^{-1}$ is an analytic, bijective \textit{tensor decomposition} function. We aspire to solve the tensor decomposition problem only in this well-behaved setting.

The condition (ii) is more restrictive, but allows us to tackle the tensor decomposition problem using only efficient linear algebra.
Recall that the standard $1$-flattening of $\mathcal{A} \in \C^{\ell+1} \otimes \C^{m+1} \otimes \C^{n+1}$ consists of interpreting $\mathcal{A}$ as the matrix $\mathcal{A}_{(1)} \in \C^{(\ell+1) \times (m+1)(n+1)}$. For a rank-$1$ tensor $\mathcal{A}=\alpha \otimes \beta \otimes \gamma$ this identification is defined by 
\[
 \mathcal{A}_{(1)} = \alpha (\beta \otimes \gamma)^\top,
\]
and the general case follows by linearity. The tensor product\footnote{The tensor product is defined uniquely by the linear space into which it maps by universality \cite{Greub1978}, so we do not make a distinction in notation.} in the foregoing expression is also called the \textit{reverse-order Kronecker product} 
\[
 \otimes : \C^{m+1} \times \C^{n+1} \to \C^{(m+1)(n+1)},\; (\beta, \gamma) \mapsto [ \beta_{i} \gamma_{j} ]_{(i,j)},
\]
where the standard bases were assumed for these Euclidean spaces and the indices $(i,j)$ are sorted by the reverse lexicographic order. Note that the $1$-flattening is easy to compute when $\mathcal{A}$ is given in coordinates relative to the standard tensor product basis. In that case it suffices to reshape the coordinate array to an $(\ell+1) \times (m+1)(n+1)$ array (e.g., as in Julia's or Matlab's \texttt{reshape} function).

\subsection{The reshaping trick}\label{subsec:reshapingtrick}
For dealing with higher-order tensors $\mathcal{A}\in\C^{n_1+1}\otimes\cdots\otimes\C^{n_d+1}$ with $d > 3$, we rely on the well-known reshaping trick. It consists of interpreting a higher-order tensor as a third-order tensor. While the approach described in the remainder of the paper could be applied directly to higher-order tensors as well, the range of ranks $r$ to which this version would apply is (much) more restrictive than the range obtained from reshaping. 

Recall that reshaping $\mathcal{A}$ to a third-order tensor $\mathcal{A}_{(I,J,K)}$ is a linear isomorphism
\begin{align}
 \C^{n_1+1} \otimes \cdots \otimes \C^{n_d+1} 
 &\simeq (\otimes_{i \in I} \C^{n_i+1}) \otimes (\otimes_{j\in J} \C^{n_j+1}) \otimes (\otimes_{k\in K} \C^{n_k+1}) \nonumber \\
 &\simeq \C^{\prod_{i\in I} (n_i+1)} \otimes \C^{\prod_{j\in J} (n_j+1)} \otimes \C^{\prod_{k\in K} (n_k+1)} \tag{I}\label{eqn_isomorphism}
\end{align}
that identifies $\alpha^1 \otimes \cdots\otimes \alpha^d$ with $(\otimes_{i\in I} \alpha^i) \otimes (\otimes_{j\in J} \alpha^j) \otimes (\otimes_{k\in K} \alpha^k)$, wherein $I \sqcup J \sqcup K$ partitions $\{1,\ldots,d\}$. The general case follows by linearity from the universal property \cite{Greub1978}.

It was shown in \cite[Section~7]{COV2017} that under some conditions, the unique CPD of $\mathcal{A}$ can be recovered from the CPD of a reshaping $\mathcal{A}_{(I,J,K)}$. Here, we exploit the following slightly more general result that requires no $r$-identifiability of $\mathcal{A}$.
\begin{lemma} \label{lem_technical_bound}
Let $\mathcal{A} \in \C^{n_1+1}\otimes\cdots\otimes\C^{n_d+1}$ be a rank-$r$ tensor. If there exists a partition $I\sqcup J\sqcup K$ of $\{1,\ldots,d\}$ such that the third-order reshaping $\mathcal{A}_{(I,J,K)} \in \C^{\ell+1} \otimes \C^{m+1}\otimes\C^{n+1}$ is $r$-identifiable, then $\mathcal{A}$ is $r$-identifiable and the CPD of $\mathcal{A}_{(I,J,K)}$ is the CPD of $\mathcal{A}$ under the isomorphism in \cref{eqn_isomorphism}.
\end{lemma}
\begin{proof}
Let $I\sqcup J\sqcup K$ be such a partition and assume that $\mathcal{A}_{(I,J,K)}$'s unique CPD is $\mathcal{A}_{(I,J,K)}=\sum_{q=1}^r \alpha_q\otimes\beta_q\otimes\gamma_q$. Let $\mathcal{A} = \sum_{q=1}^r \alpha_q^1 \otimes \cdots \otimes \alpha_q^d$ be an arbitrary rank-$r$ CPD. Then, 
\[
 \mathcal{A}_{(I,J,K)} = \sum_{q=1}^r (\otimes_{i\in I} \alpha_q^i) \otimes ( \otimes_{j\in J} \alpha_q^j ) \otimes (\otimes_{k \in K} \alpha_q^k)
\]
because of \cref{eqn_isomorphism}. Since $\mathcal{A}_{(I,J,K)}$'s rank-$r$ CPD is unique by assumption, it follows there is some permutation $\pi$ and scalars $\rho_q \sigma_q \tau_q = 1$ such that we have, for all $q=1,\ldots,r$,  
\[
\rho_q \alpha_{\pi_q} = \otimes_{i\in I} \alpha_q^i,\quad 
\sigma_q \beta_{\pi_q} = \otimes_{j \in J} \alpha_q^j, \text{ and }\quad 
\tau_q \gamma_{\pi_q} = \otimes_{k \in K} \alpha_q^k.
\]
This implies that \textit{the} CPD of $\mathcal{A}_{(I,J,K)}$ under the linear isomorphism \cref{eqn_isomorphism} results in \emph{a} CPD of $\mathcal{A}$. 

It remains to show that it is the unique one. Denote the set of $r$ rank-$1$ summands in any two CPDs of $\mathcal{A}$ by $\{ \mathcal{T}_1, \ldots, \mathcal{T}_r \}$ and $\{\mathcal{T}_1', \ldots, \mathcal{T}_r'\}$ respectively. The previous paragraph showed that any CPD of $\mathcal{A}$ implies a CPD of $\mathcal{A}_{(I,J,K)}$. 
Since $\mathcal{A}_{(I,J,K)}$'s CPD is unique, there exist permutations $\pi$ and $\pi'$ such that for all $q=1,\ldots,r,$
\[
 (\otimes_{i \in I} \alpha_q^i) \otimes (\otimes_{j\in J} \beta_q^i) \otimes (\otimes_{k\in K} \gamma_q^k) = (\mathcal{T}_{\pi_q})_{(I,J,K)} = (\mathcal{T}_{\pi_q'}')_{(I,J,K)}.
\]
Since reshaping is injective and sends rank-$1$ tensors in the domain to rank-$1$ tensors in the codomain, we conclude that $\{\mathcal{T}_1,\ldots,\mathcal{T}_r\} = \{\mathcal{T}_1', \ldots, \mathcal{T}_r'\}$. That is, all of $\mathcal{A}$'s CPDs must consist of the same rank-$1$ terms, which is exactly the definition of $r$-identifiability.
\end{proof}

For decomposing higher-order tensors $\mathcal{A}$ of rank $r$ via the reshaping trick, we proceed as follows:
\begin{enumerate}
 \item Find $I \sqcup J \sqcup K =\{1,\ldots,d\}$ such that $\mathcal{A}_{(I,J,K)}$ is $r$-identifiable and its rank satisfies \eqref{eqn_rank_condition}.
 \item Compute the CPD $\mathcal{A}_{(I,J,K)} = \sum_{q=1}^r \alpha_q \otimes \beta_q \otimes \gamma_q$, e.g., as in \cref{alg:pseudoalg}.
 \item Viewing the recovered vectors as $\alpha_q \in \otimes_{i\in I} \C^{n_i+1}$, $\beta_q \in \otimes_{j\in J} \C^{n_j+1}$, and $\gamma_q \in \otimes_{k\in K} \C^{n_k+1}$, a CPD of $\mathcal{A}$ is $\mathcal{A}_1+\cdots+\mathcal{A}_r$ where the tensors $\mathcal{A}_q$ are isomorphic to $\alpha_q \otimes \beta_q \otimes \gamma_q$ under \cref{eqn_isomorphism}.
\end{enumerate}
Note that the rank-$1$ decompositions in the third step can be computed for example with a sequentially truncated higher-order singular value decomposition \cite{VVM2012}, as in our implementation, or with a cross approximation \cite{OST2008}. 

Because of the reshaping trick, we will henceforth describe our approach only for $d=3$.

\subsection{Polynomial systems defined by flattenings} \label{subsec:polfromflat}
Having delineated the range \eqref{eqn_rank_condition} in which the proposed \algoname{} algorithm will work, we continue by describing it.
Our strategy to compute the tensor rank decomposition of 
\begin{align} \tag{A} \label{eq:rankdecomp}
 \A = \sum_{i=1}^r \alpha_i \otimes \beta_i \otimes \gamma_i ~ \in ~ \C^{\ell+1} \otimes \C^{m+1} \otimes \C^{n+1}
\end{align} 
is to compute the points 
\[
\z_i = (\beta_i, \gamma_{i}) \in \C^{m+1} \times \C^{n+1}, \quad i = 1, \ldots, r,
\]
first as the solution of a system of (low-degree) polynomial equations. Thereafter we compute the $\alpha_{i}$'s with efficient linear algebra by plugging $\beta_{i}$ and $\gamma_{i}$ into \eqref{eq:rankdecomp} and solving the resulting linear system of equations. Indeed, since
\begin{align} \label{eqn_flattening}
 \A_{(1)} = \begin{bmatrix} \alpha_1 & \cdots & \alpha_r \end{bmatrix} \begin{bmatrix} \beta_1 \otimes \gamma_1 & \cdots & \beta_r \otimes \gamma_r \end{bmatrix}^\top
\end{align}
and \cref{assum:vanishingideal} guarantees that the second $r \times (m+1)(n+1)$ matrix is of full rank $r$ (see \cref{lem:zerodim} below), so it has a right inverse. Applying the latter on the right to $\A_{(1)}$ yields the corresponding $\alpha_i$'s.
Note that it suffices to compute the points $\beta_{i}$ and $\gamma_{i}$ up to a nonzero scaling factor. Therefore, it is natural to consider our problem in \emph{complex projective space}.

Recall that the $k$-dimensional complex projective space $\PP^{k}$ is the space of equivalence classes $[x] = \{ \lambda x \mid \lambda \in \C \setminus \{0\} \}$ for $x \in \C^{k+1} \setminus \{0\}$. 
The entries of the vector $x = (x_0, \ldots, x_k) \in \C^{k+1} \setminus \{0\}$ are called \textit{homogeneous coordinates} of $[x]$. With a slight abuse of notation we will write $x = (x_0 : \cdots : x_k) \in \PP^k$ for both the equivalence class $[x]$ and a set of homogeneous coordinates $x \in \C^{k+1} \setminus \{0\}$.

The proposed \algoname{} algorithm exploits the fact that the points $\z_i = (\beta_i, \gamma_i) \in \PP^m \times \PP^n$ are defined by algebraic relations on $X = \PP^{m} \times \PP^n$ that can be computed directly from the tensor $\A$. Such algebraic relations are \textit{homogeneous polynomials} in a \textit{bi-graded polynomial ring}. In this context, it is natural to think of \emph{degrees} as 2-tuples $(d,e) \in \N^2$, where $d$ is the degree in the variables corresponding to $\PP^m$ and $e$ is the degree in the variables corresponding to $\PP^n$. Note that this differs from the more familiar setting where the degree is a natural number $d \in \N$. For more on multi-graded rings, see \cite{miller2005combinatorial}. The flexibility provided by this \emph{bi-graded} setting reduces the complexity of our algorithm.
Concretely, we work in the $\N^2$-graded polynomial ring 
\begin{equation} \label{eq:grading}
S = \C[x_{0}, \ldots, x_{m}, y_{0}, \ldots, y_{n}] = \bigoplus_{(d,e) \in \N^2} S_{(d,e)},\; \text{where} \quad S_{(d,e)} = \bigoplus_{|a| = d, |b|=e} \C \cdot x^{a} y^{b}.
\end{equation}
Here we used the notation $x^{a} = x_{0}^{a_{0}} \cdots x_{m}^{a_{m}}$ and $|a| = a_{0} + \cdots + a_{m}$ for $a = (a_{0}, \ldots, a_{m}) \in \N^{m + 1}$, and analogously for $b$. The \emph{graded pieces} $S_{(d,e)}$ are vector spaces over $\C$. The variables $x$ correspond to homogeneous coordinates on the factor $\PP^{m}$ and in $S_{(d,e)}$, $d$ is the degree in the $x$-variables. Analogously, elements of $S_{(d,e)}$ have degree $e$ in the $y$-variables, which are related to $\PP^n$. 

An element $f \in S$ is called \emph{homogeneous} if $f \in S_{(d,e)}$ for some $(d,e)\in \N^2$.
The key reason that the ring $S$, with its grading \eqref{eq:grading}, is naturally associated to $X$ is that \emph{homogeneous elements $f \in S_{(d,e)}$ have well-defined zero sets on $X$}. By this we mean that for $f \in S_{(d,e)}$ and for any $\z = (x , y) \in X$, $f(x,y) = 0$ is independent of the choice of homogeneous coordinates. Indeed, this follows from 
\begin{equation} \label{eq:actiononS}
f(\lambda x , \mu y) = \lambda^d  \mu^e f(x,y) \quad \text{for} \quad \lambda, \mu \in \C \setminus \{0\}.
\end{equation}
Therefore, whenever $f$ is homogeneous it makes sense to write $f(\z) = 0$ if $f$ vanishes on some set of homogeneous coordinates for $\z$, and to define the \emph{subvariety of $X$ defined by $f$} as $V_X(f) = \{ \z \in X ~|~ f(\z) = 0\}$. For a homogeneous ideal $I \subset S$ (i.e., $I$ is generated by homogeneous elements with respect to the grading \eqref{eq:grading}), we denote the \emph{subvariety of $X$ corresponding to $I$} by 
\[
V_X(I) = \{ \z \in X ~|~ f(\z) = 0, \text{ for all homogeneous } f \in  I \}.
\]
That is, $V_X(I)$ contains the solutions of the polynomial system defined by the simultaneous vanishing of all homogeneous equations $f \in I$. 

\begin{example}
The polynomials $f_1, \ldots, f_5$ in \cref{ex:running1} are homogeneous elements of degree $(1,1)$ in $S$, i.e., $f_i \in S_{(1,1)}$. They generate the homogeneous ideal $I$, whose corresponding subvariety is $V_X(I) = \{\z_1,\ldots,\z_4\} \subset X = \PP^2 \times \PP^2$.
\end{example}

With this notation in place, we turn back to the tensor $\A$ and show that, under suitable assumptions, $V_X(I) = \{\z_1,\ldots,\z_r\}$ for some homogeneous ideal $I$ defined by $\A$. This will generalize the procedure in \cref{ex:running1}. 
Consider the flattening $\A_{(1)}$ from \cref{eqn_flattening}. This flattening has a natural interpretation as a $\C$-linear map 
\[
 \A_{(1)} :~ S_{(1,1)} ~\longrightarrow~ \C^{\ell+1}.
\]
Indeed, we can identify the space $\C^{m+1} \otimes \C^{n+1}$ with the graded piece $S_{(1,1)}$ of degree $(1,1)$ by 
\begin{equation} \label{eq:monbasis}
e_{k} \otimes e_{l} ~ \longmapsto~ x_k y_l, \qquad 0 \leq k \leq m,\; 0 \leq l \leq n,
\end{equation}
where $e_{k}$ represents the $(k+1)$-st standard basis vector of $\C^{m + 1}$, and analogously for $e_{l}$ and $\C^{n+1}$.
For brevity, we also write $\A_{(1)}$ for a matrix representation of the map $\A_{(1)}$, where the standard basis of monomials \eqref{eq:monbasis} is used for $S_{(1,1)}$. The $\ell+1$ rows of $\A_{(1)}$ are elements of the dual space $(S_{(1,1)})^\vee$ represented in its dual basis 
\begin{equation} \label{eq:dualbasis}
\left \{ \d_{kl} = \frac{\d^2}{\d {x_k} \d{y_l}} \mid 0 \leq k \leq m,~  0 \leq l \leq n \right \}.
\end{equation} 
From this vantage point, the rows of the second factor in \cref{eqn_flattening} are
\begin{equation*}
(\beta_i \otimes \gamma_i)^\top = \sum_{\substack{0 \leq k \leq m \\ 0 \leq l \leq n}} \beta_{i,k} \gamma_{i,l} \d_{kl} ~ \in S_{(1,1)}^\vee.
\end{equation*}
It is clear that for any $f \in S_{(1,1)}$, we have $(v \otimes w)^\top (f) = f(v,w)$. 

Let $f_1, \ldots, f_s \in S_{(1,1)}$ be a $\C$-basis for the kernel $\ker \A_{(1)}$.
The $f_i$ generate a homogeneous ideal $I = \ideal{\ker \A_{(1)}} = \ideal{f_1,\ldots,f_s} \subset S$. The following lemma characterizes precisely what we mean by \emph{generic} in \cref{assum:vanishingideal}.

\begin{lemma} \label{lem:zerodim}
Consider the space $\C^{\ell+1} \otimes \C^{m+1} \otimes \C^{n+1}$ with $\ell \ge m \ge n > 0$. 
Then, for all $r \le \min\left\{ \ell+1, mn \right\}$
there exists a Zariski dense open subset $U \subset \mathcal{S}_r$ such that for all tensors $\A \in U$, 
\begin{enumerate}
\item[(i)] $\A$ is $r$-identifiable,
\item[(ii)] the flattening $\mathcal{A}_{(1)}$ has rank $r$,
\item[(iii)] the ideal $I = \ideal{\ker \A_{(1)}} \subset S$ is such that $V_X(I)$ consists of the points $\{\z_1, \ldots, \z_r\}$, and these points have multiplicity one.
\end{enumerate}
\end{lemma}
\begin{proof}
Items (i)-(ii) hold on a dense open subset $U_1$ by \cite[Proof of Proposition 8.1]{BCO2013}. Item (iii) holds on a dense open subset $U_2$ by the scheme-theoretic version of the Trisecant Lemma \cite[Proposition 1.4.3]{russo2016geometry}. Our subset $U$ is $U_1 \cap U_2$. 
\end{proof}

Point (iii) in \cref{lem:zerodim} is a technicality we will need in \cref{sec:normalforms} in order to invoke the results from \cite[Section 5.5]{telen2020thesis}. For the reader who is familiar with algebraic geometry, we included a consequence of \cref{lem:zerodim} relating $I$ to the vanishing ideal of $V_X(I)$ in \cref{cor:satisvan}. 

We were not able to construct an example for which (i)+(ii) are satisfied, but (iii) is not. This raises the question whether (i)+(ii) implies (iii), and if not, how our method extends to counterexamples. 

\subsection{The high-level algorithm} \label{sec_sub_highlevel}
We conclude the section by presenting a high-level pseudo-algorithm for computing the tensor rank decomposition \eqref{eq:rankdecomp} of $\A$. This is presented as \cref{alg:pseudoalg}. Its steps summarize the discussion up to this point. 

\begin{algorithm} 
\caption{Compute the tensor rank decomposition of $\A \in \C^{\ell+1}  \otimes  \C^{m+1} \otimes\C^{n+1}$}
\begin{algorithmic}[1]
\REQUIRE $\A$ satisfies \cref{assum:vanishingideal}.
\STATE $\A_{(1)} \leftarrow$ the $(\ell+1) \times (m+1)(n+1)$ $1$-flattening of $\A$ 
\STATE $\{f_1,\ldots,f_s\} \leftarrow$ a $\C$-basis for $\ker \A_{(1)}$ \label{line_Cbasisker}
\STATE $I \leftarrow$ the ideal $\ideal{f_1,\ldots,f_s} \subset S$
\STATE $\{ (\beta_i,\gamma_i) \}_{i=1, \ldots,r} \leftarrow$ homogeneous coordinates for $\{ \z_i \}_{i=1,\ldots,r} = V_X(I) \subset X$ \label{line:solvepolsys}
\STATE $ (\alpha_1, \ldots, \alpha_r) \leftarrow$ solve the linear system defined by \cref{eqn_flattening}
\RETURN $\{ (\alpha_i,\beta_i,\gamma_i) \}_{i=1, \ldots,r}$
\end{algorithmic}
 \label{alg:pseudoalg}
\end{algorithm}

\begin{remark}
In a practical implementation of \cref{alg:pseudoalg}, the necessary conditions from \cref{lem:zerodim} would be handled as follows. The condition (i) that $\mathcal{A}$ is $r$-identifiable would be treated as a non-checked precondition. After obtaining the algorithm's output, the user should verify that $\mathcal{A}\approx\sum_{i=1}^r \alpha_i\otimes\beta_i\otimes\gamma_i$. If this holds, then specific $r$-identifiability can be checked using a posteriori certifications such as those in \cite{COV2017,Kruskal1977,DdL2013b,domanov2017canonical,SB2000}.
In \cref{alg:pseudoalg}, condition (ii) is checked in line \ref{line_Cbasisker} and condition (iii) in line \ref{line:solvepolsys}, see \cref{sec_sub_hnfalgorithm}.
\end{remark}

The main focus of this paper is dealing with step \ref{line:solvepolsys} of \cref{alg:pseudoalg}. We will employ a state-of-the-art method for solving polynomial equations on $X$, based on \emph{homogeneous normal forms}. This strategy is described in the next section.

\section{From polynomial equations to eigenvalues} \label{sec:normalforms}
In this section, we employ tools from computational algebraic geometry for computing points in a product of projective spaces in step \ref{line:solvepolsys} of \cref{alg:pseudoalg}. 
The next subsection summarizes the relevant results in the present setting. We show in \cref{sec_sub_mainthm} that the solutions can be obtained from an eigenvalue problem defined by \emph{pre-normal forms}. How to compute the latter is explained in \cref{sec_sub_prenormalform}. Thereafter, in \cref{sec_sub_pba}, we demonstrate that so-called \textit{pencil-based algorithms} for decomposing tensors of very low rank are closely related to our \algoname{} algorithm. The full algorithm for performing step \ref{line:solvepolsys} of \cref{alg:pseudoalg} is presented in \cref{sec_sub_hnfalgorithm}. Finally, we conclude this section with some examples.

\subsection{The eigenvalue theorem} \label{sec_sub_mainthm}
Our algorithm is built on a multi-homogeneous version of the \emph{eigenvalue theorem} (\cref{thm:MHNF} below), which allows us to find solutions of systems of polynomial equations via eigenvalue computations. Behind this is the theory of \emph{homogeneous normal forms}. In our context, these are $\N^2$-graded versions of \emph{truncated normal forms}, introduced in \cite{telen2018solving}, and special cases of the more general toric homogeneous normal forms used by \cite{telen2019numerical,bender2020toric} and formally introduced in \cite{telen2020thesis}. For our purpose, it suffices to work with slightly simpler objects, called \emph {pre-normal forms}, so we use homogeneous normal forms only implicitly. 
For full proofs and more details we refer to \cite[Section 5.5.4]{telen2020thesis}.

Consider the space $X = \PP^m \times \PP^n$ and its associated ring $S$, which is the $\N^2$-graded polynomial ring from \eqref{eq:grading}.
Let $I \subset S$ be a homogeneous ideal. The ideal $I$ inherits the grading on $S$: 
\[ 
I = \bigoplus_{(d,e) \in \N^2} I_{(d,e)}, \qquad \text{where } \qquad I_{(d,e)} = I \cap S_{(d,e)}
\]
is a $\C$-subspace of $S_{(d,e)}$, and the quotient ring 
\[
S/I = \bigoplus_{(d,e) \in \N^2} (S/I)_{(d,e)}= \bigoplus_{{(d,e)} \in \N^2}  S_{(d,e)}/ I_{(d,e)}
\]
(these are quotients of vector spaces) similarly inherits this grading. 
Our objective is to compute a set of homogeneous coordinates of the points in $V_X(I)$ by using linear algebra computations. To accomplish this, it is necessary to work in graded pieces of $S$, $I$, and $S/I$. Let $M$ denote either of the latter.
The \emph{(multi-graded) Hilbert function} $\HF_M : \N^2 \rightarrow \N$ is given by 
\[
\HF_M(d,e) = \dim_\C M_{(d,e)}
\]
and keeps track of the dimension of the vector space $M_{(d,e)}$. Note that for a homogeneous ideal $I \subset S$ and $(d,e) \in \N^2$, we have 
$
\HF_{S/I}(d,e) = \HF_S(d,e) - \HF_I(d,e)$.

\begin{example} \label{ex:hfS}
The Hilbert function of the ring $S$ is given explicitly by 
\[
\HF_S(d,e) = \dim_\C S_{(d,e)} = \begin{pmatrix}
m + d \\ d
\end{pmatrix} \begin{pmatrix}
n + e \\ e
\end{pmatrix}  \quad \text{with} \quad (d,e) \in \N^2 . \qedhere\]
\end{example}

Before stating the main result of this section, \cref{thm:MHNF}, we present three auxiliary lemmas.  
\Cref{lem:h0,lem:hf} are well-known in commutative algebra. For lack of a precise reference and completeness, we included short proofs in \cref{app:proofs}.
\begin{lemma} \label{lem:h0}
Let $I \subset S$ be a homogeneous ideal such that $V_X(I)$ consists of $r < \infty$ points $\{\z_1, \ldots, \z_r\}$. For each $(d',e') \in \N^2$, there exists a homogeneous polynomial $h_0 \in S_{(d',e')}$ such that $V_X(h_0) \cap V_X(I) = \emptyset$. Equivalently, we can find $h_0 \in S_{(d',e')}$ such that $h_0(\z_i) \neq 0,\; i = 1, \ldots, r$. 
\end{lemma}

\begin{lemma} \label{lem:hf}
Let $I \subset S$ be a homogeneous ideal such that $V_X(I)$ consists of $r < \infty$ points, each with multiplicity one. There exists $(d,e) \in \N^2$ with $(d,e) \neq (1,1)$ and $(d-1,e-1) \geq (0,0)$ (entry-wise) such that $\HF_{S/I}(d,e) = r$. 
\end{lemma}

\begin{assumption} \label{assum:2}
Henceforth, we assume the following:
\begin{enumerate}
\item Let $I = \ideal{\ker \A_{(1)}} = \ideal{f_1,\ldots,f_s}$ be as in \cref{subsec:polfromflat}, where $f_1, \ldots, f_s \in S_{(1,1)}$ form a basis for $\ker \A_{(1)}$. We assume that $\A$ satisfies \cref{assum:vanishingideal}, so that $V_X(I)$ consists of the $r$ points $\z_1, \ldots, \z_r$ with multiplicity one by \cref{lem:zerodim}.
\item The tuple $(d,e) \in \N^2$ is such that $(d,e) \geq (1,1)$ and $\HF_{S/I}(d,e) = r$. Note that this is satisfied for $(d,e) = (1,1)$ by construction of $I$.\footnote{Often we will require $(d,e) \neq (1,1)$ and $(d-1,e-1)\geq (0,0)$. This makes sense by \cref{lem:hf}. The problem of how to find such a tuple $(d,e)$ will be the topic of \cref{sec:regorder3}.}
\item We write $(d',e') = (d - 1, e-1)$, and $h_0 \in S_{(d',e')}$ is such that $h_0(\z_i) \neq 0, i = 1, \ldots, r$. This makes sense by \cref{lem:h0}.
\end{enumerate}
\end{assumption}

\begin{lemma} \label{lem:towardsmhnf}
Let $I, (d,e)$ and $h_0$ be as in \cref{assum:2}. If a $\C$-linear map $N: S_{(d,e)} \rightarrow \C^r$ satisfies 
\begin{equation} \label{eq:MHNFcondition}
 \rank(N) = r \quad\text{and}\quad \ker N = I_{(d,e)}, 
 \end{equation}
then the induced linear map $N_{h_0} : S_{(1,1)} \rightarrow \C^r, f \mapsto N(h_0 f)$ has rank $r$. 
\end{lemma}
\begin{proof}
Note that $\HF_{S/I}(1,1) = r$ by construction. By \cite[Proposition 5.5.7]{telen2020thesis}, the fact that $\HF_{S/I}(d,e) = r$ implies that $(1,1)$ and $(d',e')$ form a so-called \emph{regularity pair}. Surjectivity of the map $N_{h_0}$ then follows directly from \cite[Lemma 5.5.3]{telen2020thesis}.
\end{proof}

Linear maps $N: S_{(d,e)} \rightarrow \C^r$ satisfying \eqref{eq:MHNFcondition} will play an important role in translating our polynomial root finding problem into an eigenvalue problem. We will therefore give them a name. 
\begin{definition}[Pre-normal forms]
A $\C$-linear map $N: S_{(d,e)} \rightarrow \C^r$ satisfying \eqref{eq:MHNFcondition} is called a \emph{pre-normal form} on $S_{(d,e)}$.
\end{definition}
If $r = \ell+1$, \cref{assum:vanishingideal} implies that $\A_{(1)}: S_{(1,1)} \rightarrow \C^r$ is a pre-normal form on $S_{(1,1)}$. 

We use the terminology \emph{pre-}normal form because, following \cite[Definition 5.5.6]{telen2020thesis}, the term \emph{normal form} is reserved for linear maps satisfying an extra technical condition.
Normal forms are closely related to a classical result in computational algebraic geometry called the \emph{eigenvalue theorem} \cite[\S\,2.4]{cox2006using}, which is used for computing isolated solutions to polynomial systems. The pre-normal forms introduced above will be useful for formulating a multi-homogeneous version of this theorem, namely \cref{thm:MHNF} below. Before stating it, we need to fix some additional notation.

In \cref{lem:towardsmhnf}, $N_{h_0}$ is the composition of $N$ with the linear map that represents multiplication by $h_0 \in S_{(d',e')}$. More generally, for $g \in S_{(d',e')}$ we define $N_g : S_{(1,1)} \rightarrow S_{(d,e)}$ as $N_g(f) = N(gf)$. 
In \cref{thm:MHNF}, we will restrict the map $N_{h_0}$ to an $r$-dimensional subspace $B \subset S_{(1,1)}$ such that the resulting map is invertible. We denote this restriction by $(N_{h_0})_{|B}$. In practice, it suffices to select $r$ columns of the matrix of $N_{h_0}$ so that the resulting $r \times r$ submatrix $(N_{h_0})_{|B}$ is invertible. 

In what follows, we will write $w_i ^\top = (\beta_i \otimes \gamma_i)^\top_{|B} \in B^\vee$ for the linear functional representing evaluation at $\z_i = (\beta_i, \gamma_i)$. This is the restriction of the functional $(\beta_i \otimes \gamma_i)^\top \in S_{(1,1)}^\vee$ to the vector space $B$. Concretely, we set 
\begin{equation} \label{eq:evdef}
w_i^\top(b) = (\beta_i \otimes \gamma_i)^\top_{|B} (b) = b(\beta_i,\gamma_i), \quad \text{with} \quad b(x,y) \in B \subset S_{(1,1)}.
\end{equation}
After fixing a basis for $B^\vee$, $w_i^\top$ can be represented as a row vector.

If $g, h \in S_{(d',e')} \setminus \{0\}$ are two homogeneous polynomials of the same degree, then the fraction $g/h$ is a well-defined function on $X \setminus V_X(h)$. Indeed, the evaluation of this function does not depend on the choice of homogeneous coordinates. Therefore, we may write $(g/h)(\z)$ for the evaluation of this rational function at $\z \in X \setminus V_X(h)$.

\begin{theorem}[Eigenvalue theorem] \label{thm:MHNF}
Let $I$, $(d,e)$ and $h_0$ be as in \cref{assum:2} and let $N$ be a pre-normal form. Let $B \subset S_{(1,1)}$ be any $r$-dimensional subspace such that the restriction $(N_{h_0})_{|B} : B \rightarrow \C^r$ is invertible. For any $g \in S_{(d',e')}$ we have
\begin{equation} \label{eq:EVrelation}
w_i^\top \circ M_{g/h_0} = \frac{g}{h_0}(\z_i) \cdot w_i^\top, \quad i = 1, \ldots, r, 
\end{equation}
where $M_{g/h_0}: B \rightarrow B$ is the composition $(N_{h_0})_{|B}^{-1} \circ (N_g)_{|B}$.
\end{theorem}
\begin{proof}
This follows from \cite[Theorem 5.5.3, Propositions 5.5.4 and 5.5.5]{telen2020thesis} and the fact that $(1,1)$ and $(d',e')$ form a regularity pair (see the proof of \cref{lem:towardsmhnf}).
\end{proof}

After fixing a basis for $B$ and representing $w_i$ in the dual basis for $B^\vee$, \eqref{eq:EVrelation} is a standard matrix eigenproblem: $w_i^\top M_{g/h_0} = \lambda_i w_i^\top$. That is, $(\lambda_i,w_i)$ is a left eigenpair of the $r \times r$ matrix $M_{g/h_0}$. Note that \cref{thm:MHNF} implies that all maps of the form $M_{g/h_0}$ share a set of eigenvectors. 

We now sketch one way of using \cref{thm:MHNF} to retrieve the coordinates of $\z_i = (\beta_i, \gamma_i)$ from eigenvalues, assuming a pre-normal form $N: S_{(d,e)} \rightarrow \C^r$ is given. The problem of computing a pre-normal form is addressed in the next subsection. We assume $d \geq 2$.\footnote{If $d = 1$ and $e \geq 2$, the roles of $d$ and $e$ can be swapped so that the approach still works.}  Let $h \in S_{(d'-1,e')}$ and $h_0 \in S_{(d',e')}$ be homogeneous polynomials that do not vanish at any of the points $\z_i$. These can be chosen generically. Set $g_j = x_j h \in S_{(d',e')}, j = 0, \ldots, m$. Choose $B \subset S_{(1,1)}$ of dimension $r$ such that $(N_{h_0})_{|B}$ is invertible and compute the matrices $M_j = M_{g_j/h_0} =  (N_{h_0})_{|B}^{-1} \circ (N_{g_j})_{|B}$. By \cref{thm:MHNF}, the eigenvalues of $M_j$ are given by $\lambda_{ji} = (g_j/h_0)(\z_i)$. Writing $\beta_{ij}$ for the $j$-th coordinate of $\beta_i$, we have 
\[ 
(\lambda_{0i} : \cdots : \lambda_{mi}) = \left ( \frac{\beta_{i0} h(\z_i)}{h_0(\z_i)} : \cdots : \frac{\beta_{im} h(\z_i)}{h_0(\z_i)} \right ) = ( \beta_{i0} : \cdots : \beta_{im}) = \beta_i.
\]
Note that if $(d,e) = (2,1)$, we can take $h = 1$.

Subsequently, we compute $\gamma_i$ by solving the linear system of equations 
\begin{equation} \label{eq:linsysgamma}
f_1(\beta_i,y) = \cdots = f_s(\beta_i,y) = 0.
\end{equation}

The foregoing approach requires that $(d,e) \neq (1,1)$. Otherwise $S_{(d',e')}=S_{(0,0)} = \C$, and we can only evaluate constant functions using \cref{thm:MHNF}.

\subsection{Computing pre-normal forms} \label{sec_sub_prenormalform}
As illustrated in the previous subsection, once we have computed a pre-normal form, the points $\z_i$ can be recovered using basic linear algebra computations.
A natural next issue to address is \emph{how to compute a pre-normal form}. 

Our starting point is a basis $f_1, \ldots, f_s \in S_{(1,1)}$ of $\ker \A_{(1)}$, generating our ideal $I = \ideal{\ker \A_{(1)}} = \ideal{f_1,\ldots,f_s}$. For any tuple $(d,e) \in \N^2$ such that $(d,e) \geq (1,1)$, the degree-$(d,e)$ part $I_{(d,e)}$ of $I$ is the $\C$-vector space spanned by 
\begin{equation} \label{eq:alphabasis}
\{ x^{a'} y^{b'} f_i ~|~ (|a'|, |b'|) = (d',e'),~ i = 1, \ldots, s \} ~ \subset ~ S_{(d,e)},
\end{equation}
where $(d',e') = (d-1,e-1)$. If $d = 0$ or $e = 0$, we have $I_{(d,e)} = \{0\}$.
In analogy with \eqref{eq:exmacmtx}, we construct a matrix whose rows are indexed by the monomial basis elements of $S_{(d,e)}$ (i.e., the monomials $\{ x^ay^b ~|~ |a| = d, |b|=e \}$), and whose columns are the polynomials \eqref{eq:alphabasis} expanded in this basis. We denote this matrix by 
\begin{equation} \label{eq:resmap}
R_I(d,e) \in \C^{\HF_S(d,e)~\times~ s\, \HF_S(d',e')}.
\end{equation}
Such matrices represent \emph{graded resultant maps} in the terminology of \cite[Section 5.5.4]{telen2020thesis}. They are multihomogeneous versions of the classical \emph{Macaulay matrices} \cite{macaulay1916algebraic}. We present an explicit example below in \cref{ex:running3}.

Observe that the Hilbert function $\HF_I(d,e)$ is given by the rank of $R_I(d,e)$ and $\HF_{S/I}(d,e)$ by its corank. This follows immediately from the observation that the columns of $R_I(d,e)$ span $I_{(d,e)}$. A \emph{left nullspace matrix} $N$ of $R_I(d,e)$ represents a map $S_{(d,e)} \longrightarrow S_{(d,e)}/I_{(d,e)} \simeq \C^{\HF_{S/I}(d,e)}$. This has the following consequence. 
\begin{proposition} \label{prop:checkhf}
If $(d,e) \in \N^2$ is such that $\HF_{S/I}(d,e) = r$, then any left nullspace matrix $N$ of $R_I(d,e)$ represents a pre-normal form. 
\end{proposition} 

We conclude that a pre-normal form $N$ can be computed, for instance, from a full singular value decomposition (SVD) of $R_I(d,e)$, where $\HF_{S/I}(d,e) = r$.
This solves the problem of computing a pre-normal form, assuming that we know a degree $(d,e) \in \N^2$ for which $\HF_{S/I}(d,e) = r$. The problem of finding such degrees is addressed in \cref{sec:regorder3}.

\subsection{Relation to pencil-based algorithms} \label{sec_sub_pba}
To obtain the homogeneous coordinates for $\z_1, \ldots, \z_r$ as eigenvalues of the matrices $M_{g/h_0}$, we usually have to work with pre-normal forms on $S_{(d,e)}$, where $(d,e) \neq (1,1)$ and $(d',e') \geq (0,0)$. An exception is the case where $r \leq m+1 \leq \ell + 1$. For these tensors of very low rank, a pre-normal form $N: S_{(1,1)} \rightarrow \C^r$ will suffice under the mild condition that 
\begin{equation} \label{eq:linindepgamma}
[ \beta_1 ~ \cdots ~ \beta_r] ~ \in \C^{(m+1) \times r} \quad \text{has rank $r$.}
\end{equation}
The underlying reason is that vanishing at $\{\z_1,\ldots, \z_r\}$ gives $r$ linearly independent conditions on $S_{(1,0)}$. The proof of the following theorem is another consequence of the theory of homogeneous normal forms and is deferred to \cref{app:proofs}.

\begin{theorem}[Eigenvalue theorem for low ranks] \label{thm:MHNFlowrank}
Let $I = \ideal{\ker \A_{(1)}}$ where $\A$ has rank $r \leq m+1 \leq n+1$ and satisfies both \cref{assum:vanishingideal} and \eqref{eq:linindepgamma}. Let $h_0 \in S_{(0,1)}$ be such that $h_0(\z_i) \neq 0$ for $i = 1\ldots, r$ and let $N : S_{(1,1)} \rightarrow \C^r$ be a pre-normal form on $S_{(1,1)}$. For $g \in S_{(0,1)}$ we define 
$$ \widetilde{N_{g}} : S_{(1,0)} \rightarrow \C^r \quad \text{given by } \quad   \widetilde{N_{g}}(f) = N (gf).$$
We have that $\widetilde{N_{h_0}}$ has rank $r$. For any $r$-dimensional subspace $B \subset S_{(1,0)}$ such that the restriction $(\widetilde{N_{h_0}})_{|B} : B \rightarrow \C^r$ is invertible, the eigenvalues of $ M_{y_j/h_0} = (\widetilde{N_{h_0}})_{|B}^{-1} \circ (\widetilde{N_{y_j}})_{|B}$ are $\{\gamma_{ij}/h_0(\z_i)\}_{i = 1, \ldots, r}$.
\end{theorem}

This theorem is exploited to compute $\z_i = (\beta_i, \gamma_i)$ efficiently as follows. 
If $\A$ satisfies \cref{assum:vanishingideal} and \eqref{eq:linindepgamma}, we take a basis of the $r$-dimensional row span of $\A_{(1)}$ in \cref{thm:MHNFlowrank} as our pre-normal form. This can be obtained from a compact SVD of $\A_{(1)}$. Once the $\gamma_i$ are computed from the eigenvalues of the $M_{y_j/h_0}$, the $\beta_i$ can be obtained as in \eqref{eq:linsysgamma}. Alternatively, one can use the eigenvectors of these commuting matrices $M_{y_j/h_0}$ for $j=1,\ldots,r$ \cite[Theorem 5.5.3]{telen2020thesis}.

\Cref{thm:MHNFlowrank} is intimately related to what \cite{BBV2019} called \textit{pencil-based algorithms} for solving \cref{eqn_cpd} when the rank satisfies $r \le m+1 \le \ell+1$, such as those by \cite{LRA1993,Lorber1985,SK1990}. 
Recall that pencil-based algorithms assume that $\A \in \C^{r \times r \times (n+1)}$ is a rank-$r$ tensor.\footnote{The decomposition problem for a rank-$r$ tensor in $\C^{(\ell+1)\times(m+1)\times(n+1)}$ with $r \le m+1 \le \ell+1$ can always be reduced to this so-called \emph{concise} case \cite{Landsberg2012,BCS1997} by computing an orthogonal Tucker decomposition \cite{Tucker1966} followed by a rank-$r$ decomposition of the core tensor.} 
In addition, they assume that the $\alpha_i$ form a linearly independent set, and likewise for the $\beta_i$'s. Then, we have that the \textit{tensor contraction} of $\A$ with $h_0^\top \in (\C^{n+1})^\vee$, i.e.,
\begin{align} \label{eqn_multimult}
 h_0^\top \cdot_3 \mathcal{A} = \sum_{i=1}^r (\alpha_i \otimes \beta_i) \cdot (h_0^\top \gamma_i) =  A D_{h_0} B^\top,
\end{align}
is an invertible $r \times r$ matrix insofar as $h_0^\top \gamma_i \ne 0$. Herein, $A \in \C^{r \times r}$ (respectively $B \in \C^{r \times r}$) has the $\alpha_i$'s (respectively $\beta_i$'s) as columns, and $D_{h_0} = \diag(h_0^\top \gamma_1, \ldots, h_0^\top \gamma_r)$. Let $\widetilde{N_{h_0}} = h_0^\top \cdot_3 \A$ and $\widetilde{N_{g}} = g^\top \cdot_3 \A$ for $h_0, g \in (\C^{n+1})^\vee$. Then, we have 
\[
 M_{g/h_0} = \widetilde{N_{h_0}}^{-1} \widetilde{N_{g}} = B^{-\top} D_{h_0}^{-1} D_{g} B^\top,
\]
so that the points $\beta_i$ can be recovered uniquely from the matrix of \emph{eigenvectors} $B^\top$, provided that $h_0^\top \gamma_i \ne 0$ for all $i=1,\ldots,r$. The $\alpha_i$'s and $\gamma_i$'s can then be recovered from the $2$-flattening; see \cite{LRA1993,Lorber1985,SK1990,BBV2019,Lathauwer2006} for more details.
With the foregoing suggestive notation, it is easy to see that the matrix of $\widetilde{N_{h_0}} : f \mapsto N(fh_0)$ with respect to the standard bases is precisely \cref{eqn_multimult}. Indeed, note that since
we can take $N = \A_{(1)}$, we have $f h_0 \simeq f \otimes h_0$ and so $N(f h_0) 
 = \A_{(1)}(f \otimes h_0)$.

Pencil-based algorithms may thus be interpreted as a special case of the proposed \algoname{} algorithm based on homogeneous normal forms when $r \le m+1 \le \ell+1$. Note that because of the numerical instabilities analyzed by \cite{BBV2019} caused by extracting $\beta_i$ from the \emph{eigenvectors}, we prefer to extract the $\z_i = (\beta_i,\gamma_i)$ in a different way. We compute $\beta_i$ from the \emph{eigenvalues} of $M_{x_i/h_0}$ and the corresponding $\gamma_i$ from the linear system \cref{eq:linsysgamma}.

\subsection{The algorithm} \label{sec_sub_hnfalgorithm}
The discussion so far is distilled into \cref{alg:mhnfalg}. This algorithm implements step \ref{line:solvepolsys} of \cref{alg:pseudoalg}. Note that we dropped the tilde on top of the $N_\star$'s in lines \ref{line_start_pba}--\ref{line_end_pba} to streamline the presentation.

\begin{algorithm}[h!]
\small
\caption{Compute $V_X(I)$ for $I = \ideal{\ker \A_{(1)}}= \ideal{f_1,\ldots,f_s}$}
 \label{alg:mhnfalg}
\begin{algorithmic}[1]
\REQUIRE $r$ is the rank of $\A \in \C^{(\ell+1) \times (m+1) \times(n+1)}$ with $\ell \ge m \ge n$.
\IF {$r \leq m+1$ and \eqref{eq:linindepgamma}}
\STATE $N \gets $ an $r \times (m+1)(n+1)$ matrix representing the row space of $\A_{(1)}$, whose columns are indexed by the symbols $\d_{k,l}$ \label{line_start_pba}
\FOR{$k = 0, \ldots, m$}
\STATE $N_k \gets$ the submatrix of $N$ with columns indexed by $\{ \d_{k,j} \mid j = 0, \ldots, n \}$
\ENDFOR
\STATE ${N_{h_0}} \gets$ $c_0 N_0 + \cdots + c_m N_m$, a random $\C$-linear combination of the $N_k$
\STATE $h \gets 1$ \label{line_end_pba}
\ELSE
\STATE $(d,e) \gets$ a tuple in $\N^2$ such that $(d,e) \neq (1,1)$, $(d',e') \geq (0,0)$ and $\HF_{S/I}(d,e) = r$ \label{line_alg_degree}
\STATE Construct the resultant matrix $R_I(d,e)$
\STATE $N \gets $ an $r \times \HF_S(d,e)$ matrix representing the left nullspace of $R_I(d,e)$, whose columns are indexed by the symbols $\d_{a,b}$ with $|a| = d$, $|b| =e$. \label{line_null_space}
\FOR{$(a',b')$ such that $|a'| = d'$, $|b'| = e'$}
\STATE $N_{a',b'} \gets$ submatrix of $N$ with columns indexed by $\{ \d_{a,b} \mid a-a' \geq 0,\, b-b' \geq 0\}$
\ENDFOR
\STATE $N_{h_0} \gets \sum_{|a'|=d',|b'|=e'} c_{a',b'} N_{a',b'}$, a random $\C$-linear combination of the $N_{a',b'}$
\STATE $h \gets \sum_{|a''|=d'-1,|b'|=e'} \hat{c}_{a'',b'} x^{a''} y^{b'}$, a random element of $S_{(d'-1,e')}$ 
\FOR{$k = 0, \ldots, m$}
\STATE $N_k \gets \sum_{|a''|=d'-1, |b'|=e'} \hat{c}_{a'',b'} N_{a''+e_k,b'}$ 
\ENDFOR 
\ENDIF
\STATE $Q,R,p \gets$ QR factorization of $N_{h_0}$ with optimal column pivoting \label{line_qr}
\STATE $(N_{h_0})_{|B} \gets R[:,1,\ldots,r]$ \label{line_qr2}
\FOR{$k = 0, \ldots, m$} \label{line_mm_start}
\STATE $M_{(hx_k)/h_0} \gets (N_{h_0})_{|B}^{-1}  Q^H N_k[:,p(1,\ldots,r)]$
\ENDFOR \label{line_mm_end}
\STATE $(\beta_1, \ldots, \beta_r) \gets $ simultaneous diagonalization of  $M_{(h x_0)/h_0}, \ldots, M_{(h x_m)/h_0}$
\label{step:betasfound}
\FOR{$i = 1, \ldots, r$} 
\STATE $\gamma_i \gets$ solve $f_1(\beta_i,y) = \cdots = f_s(\beta_i,y) = 0$ for $y$ \label{line_systemker}
\STATE $(\beta_i, \gamma_i) \gets $ refine $(\beta_i, \gamma_i)$ using Newton iteration \label{line_newton}
\ENDFOR
\RETURN $(\beta_1,\gamma_1), \ldots, (\beta_r,\gamma_r)$
\end{algorithmic}
\end{algorithm}
 
The first phase of the algorithm, up to line \ref{line_qr}, constructs the pre-normal form $N$ and chooses an $N_{h_0}$. This phase depends on whether we can invoke the more efficient \cref{thm:MHNFlowrank} ($r \le m+1$) or we need the full power of \cref{thm:MHNF}. In the former case, we can take $N = \A_{(1)}$, while in the latter case we need to take $N$ equal to the left null space of $R_I(d,e)$. How we choose the degree $(d,e)$ in line \ref{line_alg_degree} is explained in \cref{sec:regorder3}.
The matrix $R_I(d,e) \in \C^{\HF_S(d,e) \times s \HF_S(d',e')}$ can be constructed efficiently column-by-column without polynomial multiplication. Indeed, by \cref{eq:alphabasis} it suffices to copy the coefficients of $f_i$ relative to the monomial basis of $S_{(1,1)}$ into the correct rows; see also \cref{ex:running3} below. The left null space $N$ can be extracted from the last $r$ columns of the $U$-factor in the SVD $R_I(d,e)=USV^H$, where $\cdot^H$ denotes the conjugate transpose. In our implementation, the matrix $N_{h_0} \in \C^{r \times \HF_S(1,1)}$ is chosen by sampling the coefficients of $h_0 \in S_{(d',e')}$ identically and independently distributed from a Gaussian distribution. With probability $1$, $h_0$ satisfies $h_0(\gamma_i) \ne 0$ for all $i$; hence, this is a valid choice of $h_0$.

The next phase of the algorithm, in lines \ref{line_qr}--\ref{line_qr2}, chooses a basis $B$. 
Although in theory \cref{thm:MHNF} enables us to choose any $B$ such that $(N_{h_0})_{|B}^{-1}$ is invertible, \cite{telen2018solving} showed that for reasons of numerical stability it is crucial to choose $B$ such that $(N_{h_0})_{|B}$ is \emph{well-conditioned}. In practice, such a subspace $B$ can be found using a QR decomposition with optimal column pivoting or by using the SVD \cite[Chapter 4]{telen2020thesis}. We stated the QR approach in \cref{alg:mhnfalg}.

The multiplication matrices are constructed straightforwardly as the formula suggests in lines \ref{line_mm_start} to \ref{line_mm_end}. Note that the upper triangular matrix $(N_{h_0})_{|B}$ does not need to be inverted explicitly, rather the system can be solved by backsubstitution.

In line \ref{step:betasfound}, the matrices $M_{(h x_k)/h_0}$ need to be simultaneously diagonalized, as we have that $M_{(h x_k)/h_0} = V^{-1} \operatorname{diag}(\beta_k) V$. We compute $V$ from a random linear combination of the $M_{(h x_k)/h_0}$'s, and then extract $\beta_k$ as the diagonal entries from the (approximately) diagonalized matrix $V^{-1} M_{(h x_k)/h_0} V$.

The system in line \ref{line_systemker} is solved efficiently by noting that the coefficients of $f_j$ can be arranged in a matrix $F_j$ of size $(m+1) \times (n+1)$, so that $f_j(x,y) = x^\top F_j y$. Hence, \eqref{eq:linsysgamma} boils down to computing the kernel of $Ay = 0$ where the rows of $A \in \C^{s \times (n+1)}$ are the row vectors $\beta_i^\top F_j$. Note that by \cref{assum:vanishingideal}, $\ker A$ is spanned by $\gamma_i$. In line \ref{line_newton} the obtained solution $(\beta_i,\gamma_i)$ is refined using standard Newton iterations. Here the pseudo-inverse of the Jacobian matrix of $f_1, \ldots, f_s$ is used. This matrix has full rank if and only if $(\beta_i, \gamma_i)$ has multiplicity one. This can be used to check condition (iii) in \cref{lem:zerodim}.
\color{black}

\begin{table}[t]
\footnotesize
\setlength{\tabcolsep}{4pt}
\centering
\begin{tabular}{c|ccccc}
\backslashbox[1mm]{\small $i$}{\small $j$}
  & 0 & 1 & 2  & 3 & $\hdots$ \\ \hline
0 & 1 & 3 & 6 & 10 & $\hdots$ \\
1 & 3 & 4 & 4 & 4 & $\hdots$ \\
2 & 6 & 4 & 4 & 4 &$ \hdots$ \\
3 & 10 & 4 & 4 & 4 & $\hdots$ \\
$\vdots$ & $\vdots$ & $\vdots$ & $\vdots$ & $\vdots$ & $\ddots$
\end{tabular}
\qquad \qquad
\begin{tabular}{c|ccccccc}
\backslashbox[0.2mm]{\small $i$}{ \small $j$}
  & 0 & 1 & 2  & 3 & 4 & 5 & $\hdots$ \\ \hline
0 &1 & 3 & 6 & 10 & 15 & 21 & $\hdots$ \\
1 & 7 & \underline{12} & 15 & 16 & 15 & 12 & $\hdots$ \\
2 & 28 & 21 & 15 & 12 & 12 & 12 &$ \hdots$ \\
3 & 84 & \underline{12} & 12 & 12 & 12 & 12& $\hdots$ \\
$\vdots$ & $\vdots$ & $\vdots$ & $\vdots$ & $\vdots$ & $\vdots$ & $\vdots$ & $\ddots$
\end{tabular}
\caption{Hilbert functions $\HF_{S/I}(i,j)$ from \cref{ex:running3} (left) and \cref{ex:bigleap} (right) for small values of $i,j$. }
\label{tab:hfex4}
\end{table}

\begin{remark}
As pointed out to us by an anonymous referee and Bernard Mourrain, it is possible to replace the left nullspace computation in line 11 by a smaller linear system of equations. This interpretation corresponds to the \emph{flat extension} of the quasi-Hankel operators in \cite{bernardi2013general}. The pre-normal form $N$ can be chosen such that $N_{h_0}$ is given by $\A_{(1)}$. This determines $N$ on the $\HF_S(1,1)$-dimensional subspace $h_0 \cdot S_{(1,1)}$. For $(a',b')$ such that $|a'| = d'$ and $|b'| = e'$, the restriction $N_{|x^{a'}y^{b'} \cdot S_{(1,1)}}$ should satisfy $N_{|x^{a'}y^{b'} \cdot S_{(1,1)}}(x^{a'}y^{b'} \cdot f_i) = 0$, $i = 1, \ldots, s$. These linear conditions determine the pre-normal form $N$ on the remaining $(\HF_S(d,e) - \HF_S(1,1))$-dimensional vector space $\HF_S(d,e)/ (h_0 \cdot S_{(1,1)})$. 

This observation allows to perform the main linear algebra computations on a matrix of size $(\HF_S(d,e) - \HF_S(1,1)) \times s\, \HF_S(d',e')$, which is smaller than the size of $R_I(d,e)$. Note that the number of rows is reduced by a factor of $1- \HF_S(1,1)/\HF_S(d,e) \approx 1 - \frac{1}{m^{d-1} n^{e-1}}$. Note that the resulting pre-normal form $N$ does not represent an orthogonal projection along $I_{(d,e)}$, and we expect that this may impact the numerical accuracy. The implementation and further investigation of this procedure are beyond the scope of this paper.
\end{remark}

\subsection{Some examples} We now present two illustrative examples. The first one shows how to use the techniques explained above on the tensor $\A$ in \cref{ex:running1}.
\begin{example}[\Cref{ex:running1}, continued] \label{ex:running3}
The Hilbert function of $S/I$, where $I$ is generated by the five $f_i$'s from \cref{ex:running1}, is shown in \cref{tab:hfex4} for small degrees. From $\HF_{S/I}((1,1) + (d',e')) = r = 4$ for $(d',e') \in \N^2$ we see that every degree $(d,e) \geq (1,1)$ leads to a pre-normal form.
Using $(d',e') = (1,0)$, we obtain the pre-normal form $N$ as the cokernel of $R_I(2,1)\in\R^{18 \times 15}$, whose transpose is
\footnotesize
\[ \makeatletter\setlength\BA@colsep{4.5pt}\makeatother
\begin{blockarray}{ccccccccccccccccccc}
& \rotbra{x_0^2y_0} & \rotbra{x_0^2y_1}&\rotbra{x_0^2y_2}&\rotbra{x_0x_1y_0}&\rotbra{x_0x_1y_1}&\rotbra{x_0x_1y_2}&\rotbra{x_0x_2y_0}&\rotbra{x_0x_2y_1}&\rotbra{x_0x_2y_2}&\rotbra{x_1^2y_0}&\rotbra{x_1^2y_1}&\rotbra{x_1^2y_2}&\rotbra{x_1x_2y_0}&\rotbra{x_1x_2y_1}&\rotbra{x_1x_2y_2}&\rotbra{x_2^2y_0}&\rotbra{x_2^2y_1}&\rotbra{x_2^2y_2} \\
\begin{block}{c[cccccccccccccccccc]}
x_0f_1 & & & & \matminus1 & 1 & & & & & & & & & & & & & \\
x_1f_1 & & & & & & & & & & \matminus1 & 1 & & & & & & & \\
x_2f_1 & & & & & & & & & & & & & \matminus1 & 1 & & & & \\
x_0f_2 & & & \matminus1 & \matminus1 & & 1 & & & & & & & & & & & & \\
x_1f_2 & & & & & & \matminus1 & & & & \matminus1 & & 1 & & & & & & \\
x_2f_2 & & & & & & & & & \matminus1 & & & & \matminus1 & & 1 & & & \\
x_0f_3 & \matminus2 & & & & & & 1 & & & & & & & & & & & \\
x_1f_3 & & & & \matminus2 & & & & & & & & & 1 & & & & & \\
x_2f_3 & & & & & & & \matminus2 & & & & & & & & & 1 & & \\
x_0f_4 & & \matminus1 & & & & & & 1 & & & & & & & & & & \\
x_1f_4 & & & & & \matminus1 & & & & & & & & & 1 & & & & \\
x_2f_4 & & & & & & & & \matminus1 & & & & & & & & & 1 & \\
x_0f_5 & & & \matminus2 & & & & & & 1 & & & & & & & & & \\
x_1f_5 & & & & & & \matminus 2 & & & & & & & & & 1 & & & \\
x_2f_5 & & & & & & & & & \matminus2 & & & & & & & & & 1\\
\end{block}
\end{blockarray}.
\]
\normalsize
The missing entries represent zeros. The row indexed by $x_2f_3$ has entry $-2$ in the column indexed by $x_0x_2y_0$ and $1$ in the column indexed by $x_2^2y_0$. This comes from $x_2f_3 = -2x_0x_2y_0 + x_2^2y_0$.
The cokernel of $R_I(2,1)$ can be obtained, for instance, from the full SVD. We set  $h_0 = x_0 + x_1 + x_2$ and use the subspace $B$ spanned by $\B= \{x_0y_0, x_0y_1, x_0y_2, x_1y_0\}$. The numerical approximations of the eigenvalues of $M_{x_0/h_0}, M_{x_1/h_0}$ and $M_{x_2/h_0}$, found using Julia, are the rows of
\begin{center}
\begin{minipage}{0.6\textwidth}
\begin{verbatim}
 -1.03745e-16  0.25   0.333333      0.5
  1.0          0.25  -2.48091e-16  -3.16351e-16
 -1.64372e-16  0.5    0.666667      0.5
\end{verbatim}
\end{minipage}
\end{center}
These approximate the evaluations of $x_i/h_0$ at $\z_4, \z_3,\z_1,\z_2$ (in that order, from left to right). Consequently, the columns in the display above are homogeneous coordinates for $\beta_4, \beta_3,\beta_1,\beta_2$. The $\gamma_i$'s can then be obtained by solving the linear system \eqref{eq:linsysgamma} of 5 equations in 3 unknowns. The left eigenvectors of the matrices $M_{x_j/h_0}$ are the columns of 
\begin{center}
\begin{minipage}{0.7\textwidth}
\begin{verbatim}
  1.10585e-17   1.58104e-15   1.0          -9.8273e-16
  1.69823e-17  -2.38698e-15   2.02579e-15  -1.0
 -1.42128e-16  -1.0           1.1188e-15   -3.681e-16
  1.0           3.81158e-17  -6.61522e-16   6.55477e-16
\end{verbatim}
\end{minipage}
\end{center}
corresponding to evaluation (up to scale) of $\B$ at $\z_4,\z_3,\z_1,\z_2$.
\end{example}

The ideal $I$ in the previous example has the property that $\HF_{S/I}(1+d',1+e') = r$ for all $(d',e') \in \N^2$. Our next example shows that this is not the case in general. 

\begin{example}[A format for which $\HF_{S/I}(2,1) \neq r$] \label{ex:bigleap}
In \cref{ex:running3} we could take any $(d',e') \in \N^2$ to compute a pre-normal form. However, it may be necessary to take bigger leaps in $\N^2$ such that $\HF_{S/I}((1,1) + (d', e')) = r$. As a concrete example, consider a rank-12 tensor $\A \in \C^{12} \otimes \C^7 \otimes \C^{3}$ with the decomposition 
\[ 
\A = \sum_{i=1}^{12} \alpha_i \otimes \beta_i \otimes \gamma_i,
\]
where $\beta_1, \ldots, \beta_{12}$ are the columns of a generic $7 \times 12$ matrix, $\gamma_1, \ldots, \gamma_{12}$ are the columns of a generic $3 \times 12$ matrix and $\alpha_1, \ldots, \alpha_{12}$ are the columns of any invertible $12 \times 12$ matrix. The Hilbert function of $S/I$ where $I = \ideal{\ker \A_{(1)}}$ is shown, for small degrees, in \cref{tab:hfex4}.
By \cref{prop:checkhf}, a possible choice for $(d',e')$ is $(2,0)$; this is underlined in the left part of \cref{tab:hfex4} . Other examples are $(2,1), (2,2), (1,2), (0,4)$. Some noteworthy non-examples are $(1,0), (1,1), (0,1)$. In \cref{sec:regorder3}, we investigate choices of $(d',e')$ of the form $(d',0)$. Our results will explain why, in this example, $d' = 1$ does not work, but $d' = 2$ does.   
\end{example}

\section{Regularity} \label{sec:regorder3}
One key step of the proposed \algoname{} algorithm has not been investigated. As explained in the previous section (see also line \ref{line_alg_degree} of \cref{alg:mhnfalg}), we should determine a correct degree $(d,e)$. This choice has a major impact on the computational complexity of the proposed algorithm. Indeed, it determines the dimensions of the graded resultant matrix $R_I(d,e)$ from \cref{eq:resmap} whose left nullspace is required. 
The goal of this section is determining which degree $(d,e)$ is needed for \cref{thm:MHNF} to apply. From this we can then deduce our algorithm's computational complexity. 

As before, let $\A$ be a tensor as in \cref{eq:rankdecomp} that satisfies \cref{assum:vanishingideal}, and let the $\N^2$-graded ring from \eqref{eq:grading} be denoted by $S = \C[x_0, \ldots, x_m, y_0, \ldots, y_n]$. We assume that $r>m+1$, for otherwise \cref{thm:MHNFlowrank} applies and no choice of $(d,e)$ is required.

To compute $\beta_i$ and $\gamma_i$ via \cref{thm:MHNF}, we need to compute a pre-normal form on $S_{(d,e)}$ for $(d,e) \neq (1,1)$ and $(d',e') \geq (0,0)$. \Cref{prop:checkhf} tells us that we must find such a tuple $(d,e)$ for which additionally $\HF_{S/I}(d,e) = r$, where $I$ is the ideal $\ideal{ \ker \A_{(1)}}$. Motivated by this, we make the following definition. 

\begin{definition} \label{def:reg}
For a homogeneous ideal $J \subset S$ such that $J = \ideal{J_{(1,1)}}$ and $V_X(J)$ consists of $r$ points with multiplicity one, we define the \emph{regularity} of $J$ to be the set
$$ \Reg(J) = \{(d,e) \in \N^2 ~|~ (d,e) \geq (1,1) \text{ and } \HF_{S/J}(d,e) = r\}.$$
\end{definition}

Hence, our task is to find a tuple $(d,e) \in \Reg(I) \setminus \{(1,1)\}$. Recall that such a tuple exists by \cref{lem:hf}. 
In this section, for given $\ell, m, n$ and $r$ satisfying \eqref{eqn_rank_condition}, we conjecture an explicit formula for $d$ and $e$ so that $(d,1), (1,e) \in \Reg(I) \setminus \{(1,1)\}$ for generic tensors of this format and rank. We prove it in many practical cases.

Because the results in this section are of independent interest for solving structured, overdetermined systems of polynomial equations, we formulate them in a slightly more general context. The first statement of this section, \cref{prop:deto11}, is concerned with homogeneous ideals $J$ of $S$ that are generated by elements of degree $(1,1)$. After that, we specialize to a particular type of such $(1,1)$-generated ideals. More precisely, to a tuple $Z = (\z_1, \ldots, \z_r) \in X^r$ we associate an ideal $J(Z)$ which is generated by elements of degree $(1,1)$, and we investigate its Hilbert function (\cref{cor:descIperp}). In our tensor setting, we will have $J(Z) = I = \ideal{\ker \A_{(1)}}$. After pointing out in \cref{lem:Jdefpoints} that for $r \leq mn$, most configurations $Z = (\z_1, \ldots, \z_r) \in X^r$ lead to an ideal $J(Z)$ such that $V_X(J(Z)) = \{ \z_1, \ldots, \z_r \}$, where each of the $\z_i$ occurs with multiplicity one, we use \cref{cor:descIperp} to characterize $\Reg(J(Z))$ in \cref{thm:boundr,thm:conjcases,conj:reg}.

For $a = (a_0, \ldots, a_m) \in \N^{m+1}$, $b = (b_0, \ldots, b_n) \in \N^{n+1}$ we write $\d_{a,b} : S \rightarrow S$ for the differential operator
\[
\d_{a,b} = \frac{1}{a_0! \cdots a_{m}! b_0! \cdots b_{n}!} \frac{\d^{|a| + |b|}}{\d x_0^{a_0} \cdots \d x_{m}^{a_{m}} \d y_0^{b_0} \cdots \d y_{n}^{b_{n}}},
\]
such that the basis of $S_{(d,e)}^\vee$ dual to $\{x^a y^b ~|~ |a| = d, |b| = e \}$ is given by $ \{ \d_{a,b} ~|~ |a| = d, |b| = e \}$. We write $e_k$ for the standard basis vector $(0, \ldots, 1, \ldots, 0)$ with a $1$ in the $(k+1)$-st position, such that
\[
\d_{e_k,e_l} = \frac{\d^2}{\d x_k \d y_l}, \qquad 0 \leq k \leq m,~ 0 \leq l \leq n.
\]
Note that this differs from the shorter notation $\d_{kl}$ used in previous sections to have a general notation for derivatives of arbitrary order.
For $(d,e) \in \N^2$ and $J \subset S$ a homogeneous ideal, the linear space $J_{(d,e)}^\perp \subset S_{(d,e)}^\vee$ is defined as 
\[
J_{(d,e)}^\perp = \{v \in S_{(d,e)}^\vee ~|~ v(f) = 0 \text{ for all } f \in I \}.
\]
It follows from basic linear algebra that $J_{(d,e)}^\perp \simeq (S/J)_{(d,e)}$, such that $\HF_{S/J}(d,e) = \dim_\C J_{(d,e)}^\perp$. As before, we denote $(d',e') = (d-1,e-1)$. In order to simplify the notation in this section, we will use the abbreviations 
\[ 
\sumab = \sum_{\substack{|a| = d \\ |b| = e}}, \qquad \sumapbp = \sum_{\substack{|a'| = d' \\ |b'| = e'}}, \qquad \sumkl = \sum_{\substack{0 \leq k \leq m \\ 0 \leq l \leq n}}.
\]
We focus on ideals $J$ that are generated by elements of degree $(1,1)$. In this case, a functional belongs to $J_{(d,e)}^\perp$ if and only if it induces functionals in $J_{(1,1)}^\perp$.

\begin{proposition} \label{prop:deto11}
Let $J \subset S$ be a homogeneous ideal such that $J = \ideal{J_{(1,1)}}$. An element $v =  \sumab c_{a,b} \d_{a,b} \in S_{(d,e)}^\vee$ is contained in $J_{(d,e)}^\perp$ if and only if 
\[
\sumkl c_{a'+e_k, b' + e_l} \d_{e_k,e_l} \in J_{(1,1)}^\perp \quad  \text{ for all } (a',b') \text{ such that } |a'| = d', |b'| = e'.
\]
\end{proposition}
\begin{proof}
Since $J = \ideal{J_{(1,1)}}$, an element $v = \sumab c_{a,b} \d_{a,b} \in S_{(d,e)}^\vee$ is contained in $J_{(d,e)}^\perp$ if and only if $v(hf) = 0$ for all $f \in J_{(1,1)}$ and all $h \in S_{(d',e')}$. Using Leibniz' rule, we find 
\[
0 =  v(hf) = \sumab c_{a,b} \d_{a,b} (hf) = \sumab c_{a,b} \sumkl \d_{a-e_k,b-e_l}(h) ~ \d_{e_k,e_l}(f),
\]
with the convention that $\d_{a,b} = 0$ whenever $\min (a) < 0$ or $\min (b) < 0$. 
Regrouping the terms in this expression gives 
\[ 
0 = \sumapbp \d_{a',b'}(h) \sumkl c_{a'+e_k, b' + e_l} \d_{e_k,e_l}(f) \quad \text{ for all $h \in S_{(d',e')}, f \in J_{(1,1)}$}.
\]
This proves the statement.
\end{proof}

In our tensor rank decomposition setting, we are mainly interested in investigating the Hilbert function for $(1,1)$-generated ideals defined by point configurations in $X = \PP^m \times \PP^n$. To that end, fix $r$ points $Z = (\z_1, \ldots, \z_r) \in X^r$ and let $\z_i = (\beta_i, \gamma_i)$. We denote $w_i = (\beta_i \otimes \gamma_i)^\top \in S_{(1,1)}^\vee$ such that $w_i(f) = f(\beta_i,\gamma_i)$ for $f \in S_{(1,1)}$.\footnote{The notation is similar to \cref{sec:tensortopol}. Here we omit the restriction to the subspace $B$ (or, equivalently, we take $B = S_{(1,1)}$). We also drop the transpose on $w_i$ as we will think of them as column vectors instead of row vectors in this section.} In coordinates, the $w_i$ are
$ w_i = \sumkl \beta_{ik} \gamma_{il} \d_{e_k,e_l}.$
To the point configuration $Z$ we associate an ideal $J(Z) \subset S$ by setting
\begin{equation*} \label{eq:defJ}
J(Z)_{(1,1)}^\perp = \Span_\C( w_1, \ldots, w_r) \quad \text{and} \quad J(Z) = \ideal{ J(Z)_{(1,1)} }.
\end{equation*}
Note that the ideal $I = \ideal{\ker{\A_{(1)}}}$ from previous sections arises in this way.\footnote{For the reader who is familiar with algebraic geometry, we note that this is our motivation for associating the ideal $J(Z)$ to $Z$, instead of the usual vanishing ideal of the points in $Z$. These are different ideals, as $J(Z)$ is usually not saturated with respect to the irrelevant ideal of $S$.} We denote $W \subset X^r$ for the Zariski-open subset in which $w_1, \ldots, w_r$ are $\C$-linearly independent. If $r \leq \HF_S(1,1)$, $W \subset X^r$ is non-empty, and therefore dense in the Euclidean topology.\footnote{This follows from the fact that the Segre variety is not contained in a hyperplane.}  
We have the following consequences of \cref{prop:deto11}.
\begin{corollary} \label{cor:descIperp}
Let $Z = (\z_1, \ldots, \z_r) \in X^r$ and let $w_1, \ldots, w_r$ and $J(Z)$ be as above. For the maps $\iota: S_{(d,e)}^\vee \hookrightarrow S_{(1,1)}^\vee \otimes S_{(d',e')}^\vee$ and $\M: (S_{(d',e')}^\vee)^r \rightarrow S_{(1,1)}^\vee \otimes S_{(d',e')}^\vee$ given by
\[
\iota \left ( \sumab c_{a,b} \d_{a,b} \right ) = \sumapbp \sumkl c_{a'+e_k, b' + e_l} \d_{e_k, e_l} \otimes \d_{a',b'}, \quad \M(v_1, \ldots, v_r) = \sum_{i=1}^r w_i \otimes v_i,
\]
we have
$ \iota \left ( J(Z)_{(d,e)}^\perp \right ) = \im \iota\, \cap \,  \im \M.$ Moreover, if $Z \in W$, then we have that $\HF_{S/J(Z)}(d,e) = \dim_\C \left( \M^{-1}(\im \iota) \right),$
where 
\[
\M^{-1}(\im \iota) =  \left \{ (v_1,\ldots,v_r) \in  (S_{(d',e')}^\vee)^r ~|~ \M(v_1,\ldots,v_r) \in \im \iota \right \}.
\]
\end{corollary}
\begin{proof}
By \cref{prop:deto11}, $ v = \sumab c_{a,b} \d_{a,b}$ is an element of $J(Z)_{(d,e)}^\perp$ if and only if for all $a',b'$ such that $|a'| = d', |b'| = e'$, there exist $v_{i,a',b'} \in \C$ such that 
\begin{equation} \label{eq:lincomb}
\sumkl c_{a'+e_k, b' + e_l} \d_{e_k,e_l} = \sum_{i=1}^r v_{i,a',b'} w_i.
\end{equation}
Writing $v_i = \sumapbp v_{i,a',b'} \d_{a',b'} \in S_{(d',e')}^\vee$ and writing \eqref{eq:lincomb} in matrix format, we see that \eqref{eq:lincomb} is equivalent to the following equality in $S_{(1,1)}^\vee \otimes S_{(d',e')}^\vee$:
\begin{equation} \label{eq:mtxeq}
 \sumapbp \sumkl c_{a'+e_k, b' + e_l} \d_{e_k, e_l} \otimes \d_{a',b'} = \sum_{i=1}^r w_i \otimes v_i.
\end{equation}

Linear independence of $w_1, \ldots, w_r$ implies that $\M$ is injective. We have that injectivity of $\iota$ and $\M$ implies along with $\HF_{S/J(Z)}(d,e) = \dim_\C  J(Z)_{(d,e)}^\perp$ that
\[
\dim_\C  J(Z)_{(d,e)}^\perp = \dim_\C \iota \left ( J(Z)_{(d,e)}^\perp \right ) = \dim_\C ( \im \iota \, \cap \, \im \M ) = \dim_\C \left( \M^{-1}(\im \iota) \right).
\]
This concludes the proof.
\end{proof}

\begin{corollary} \label{cor:keratleastr}
Let $Z = (\z_1, \ldots, \z_r) \in W \subset X^r$ and $(d, e) \geq (1,1)$. Then, $\HF_{S/J(Z)}(d,e) \geq r$. 
\end{corollary}
\begin{proof}
The statement follows from \cite[Lemma 5.5.7]{telen2020thesis}. Nevertheless, we give an instructive proof.
Let $w_i' \in S_{(d',e')}^\vee, w_i'' \in S_{(d,e)}^\vee$ be given by $w_i'(f) = f(\beta_i,\gamma_i), w_i''(g) = g(\beta_i,\gamma_i)$ for $f \in S_{(d',e')}, g \in S_{(d,e)}$. Then $\iota ( w_i'') = w_i \otimes w_i'$ and thus  
\[
\M(w_1', 0, \ldots, 0),~ \M(0, w_2', \ldots, 0), ~\ldots, ~\M(0,0,\ldots, w_r')
\]
are all contained in $\im \iota$. Therefore $\M^{-1}(\im \iota)$ contains at least $r$ linearly independent elements, so by \cref{cor:descIperp} we have $\HF_{S/J(Z)}(d,e) \geq r$. 
\end{proof}

The space $S_{(1,1)}^\vee \otimes S_{(d',e')}^\vee$ is identified with the space of matrices of size $\HF_S(1,1) \times \HF_S(d',e')$, where the rows are indexed by $\d_{e_k,e_l}$ for $0 \leq k \leq m$, $0 \leq l \leq n$ and columns are indexed by $\d_{a',b'}$ where $(a',b') \in \N^{m+1} \times \N^{n+1}$ with $|a'| = d', |b'| = e'$. For such a matrix to be contained in $\im \iota$, a collection of partial symmetry conditions needs to be satisfied. For instance, if $(a'+e_k, b'+e_l) = (a'' + e_{k'}, b'' + e_{l'})$, then the entry in the row indexed by $\d_{e_k,e_l}$ and column indexed by $\d_{a',b'}$ should be equal to the entry in the row indexed by $\d_{e_{k'},e_{l'}}$ and column indexed by $\d_{a'', b''}$ (see \cref{ex:reg1}). Matrices in $\im \iota$ are called \emph{catalecticant matrices} \cite[Definition 1.3]{iarrobino1999power}.

We can use \cref{cor:descIperp} to compute the Hilbert function of $S/J(Z)$ via a rank computation of a matrix whose entries are monomials evaluated at the points $(\beta_i,\gamma_i)$. 
This is important for our proof of \cref{thm:conjcases}. It is most easily explained by means of an example.
\begin{example} \label{ex:reg1}
Let $(m,n) = (3,2)$, $(d,e) = (2,1)$ and $r = 6$. We consider the ideal $J(Z)$ defined by the tuple $Z = (\z_1, \ldots, \z_6) \in X^6 = (\PP^3 \times \PP^2)^6$ in the ring $S =\C[x_0,x_1,x_2,x_3,y_0,y_1,y_2]$, where $\z_i = (\beta_i, \gamma_i) = ((\beta_{i0}: \beta_{i1}: \beta_{i2} : \beta_{i3}), (\gamma_{i0}: \gamma_{i1}: \gamma_{i2}))$.  As explained above, we can identify $S_{(1,1)}^\vee \otimes S_{(1,0)}^\vee$ with $12 \times 4$ matrices.
The image of $(v_1, \ldots, v_6) \in (S_{(1,0)}^\vee)^6$, with $v_i = \sum_{q = 0}^3 v_{iq} \d_{e_q,0}$, under $\M$ is 
\begin{equation} \label{eq:mtxM} \small
\makeatletter\setlength\BA@colsep{4.5pt}\makeatother
\begin{bmatrix} \beta_1 \otimes \gamma_1 & \cdots & \beta_6\otimes\gamma_6 \end{bmatrix} \begin{bmatrix}v_1^\top \\ \vdots \\ v_6^\top \end{bmatrix} = 
\begin{blockarray}{cccc}
& \cdots & w_i & \cdots \\
\begin{block}{c[ccc]}
\vdots &  & \vdots \\
\d_{e_k,e_l}& \cdots & \beta_{ik}\gamma_{il} & \cdots \\
\vdots & & \vdots \\
\end{block}
\end{blockarray}  ~
\begin{blockarray}{ccc}
\cdots & \d_{e_j,0} & \cdots \\
\begin{block}{[ccc]}
 & \vdots & \\
\cdots & v_{ij} & \cdots\\
& \vdots & \\
\end{block}
\end{blockarray}.
\end{equation}
On the other hand, the image of $ \sum_{|a| = 2, |b| = 1} c_{a,b} \d_{a,b}$ under $\iota$ is the matrix
\begin{equation} \label{eq:exampleiota} \small
  \iota \left ( \sum_{|a| = 2, |b| = 1} c_{a,b} \d_{a,b} \right ) = 
\begin{blockarray}{ccccc}
& \d_{e_0,0} & \d_{e_1,0} & \d_{e_2,0} & \d_{e_3,0}\\
\begin{block}{c[cccc]}
\d_{e_0,e_0}& c_{2e_0,e_0} & \blue{c_{e_0+e_1,e_0}} & \green{c_{e_0+e_2,e_0}} & \brown{c_{e_0+e_3,e_0}} \\
\d_{e_0,e_1} &  c_{2e_0,e_1} & \blue{c_{e_0+e_1,e_1}} & \green{c_{e_0+e_2,e_1}} & \brown{c_{e_0+e_3,e_1}} \\
\d_{e_0,e_2} &  c_{2e_0,e_2} & \blue{c_{e_0+e_1,e_2}} & \green{c_{e_0+e_2,e_2}} & \brown{c_{e_0+e_3,e_2}} \\
\d_{e_1,e_0} &  \blue{c_{e_0+e_1,e_0}} &  c_{2e_1,e_0} &  \red{c_{e_1+e_2,e_0}} &  \magenta{c_{e_1+e_3,e_0}} \\
\d_{e_1,e_1} &  \blue{c_{e_0+e_1,e_1}} &  c_{2e_1,e_1} &  \red{c_{e_1+e_2,e_1}} &  \magenta{c_{e_1+e_3,e_1}} \\
\d_{e_1,e_2} & \blue{c_{e_0+e_1,e_2}} &  c_{2e_1,e_2} &  \red{c_{e_1+e_2,e_2}} &  \magenta{c_{e_1+e_3,e_2}} \\
\d_{e_2,e_0} & \green{c_{e_0+e_2,e_0}}&  \red{c_{e_1+e_2,e_0}} &  c_{2e_2,e_0} &  \orange{c_{e_2+e_3,e_0}} \\
\d_{e_2,e_1} & \green{c_{e_0+e_2,e_1}}&  \red{c_{e_1+e_2,e_1}} &  c_{2e_2,e_1} &  \orange{c_{e_2+e_3,e_1}} \\
\d_{e_2,e_2} & \green{c_{e_0+e_2,e_2}}&  \red{c_{e_1+e_2,e_2}} &  c_{2e_2,e_2} &  \orange{c_{e_2+e_3,e_2}} \\
\d_{e_3,e_0} & \brown{c_{e_0+e_3,e_0}} &  \magenta{c_{e_1+e_3,e_0}} &  \orange{c_{e_2+e_3,e_0}} &  c_{2e_3,e_0} \\
\d_{e_3,e_1} & \brown{c_{e_0+e_3,e_1}} &  \magenta{c_{e_1+e_3,e_1}} &  \orange{c_{e_2+e_3,e_1}} &  c_{2e_3,e_1} \\
\d_{e_3,e_2} & \brown{c_{e_0+e_3,e_2}} &  \magenta{c_{e_1+e_3,e_2}} &  \orange{c_{e_2+e_3,e_2}} &  c_{2e_3,e_2} \\
\end{block}
\end{blockarray}.
\end{equation}
In order for $\M(v_1,\ldots,v_6)$, i.e., \eqref{eq:mtxM}, to be contained in $\im \iota$, it must be such that the $(\d_{e_1,e_0},\d_{e_0,0})$-entry is equal to the $(\d_{e_0,e_0},\d_{e_1,0})$-entry. This gives a linear condition on the $v_{iq}$. There are 18 such conditions. Let $v_{:,q} = (v_{1q},v_{2q},v_{3q},v_{4q},v_{5q},v_{6q})^\top$ be the column of the second matrix in \eqref{eq:mtxM} indexed by $\d_{e_q,0}$ and let 
$\Gamma = [\gamma_{ij}]$ be the matrix that has the homogeneous coordinates $\gamma_{ij}$, $j=0,\ldots,2$, of $\gamma_i$, $i=1,\ldots,6$, as columns. We also let $H_q = \diag(\beta_{1q}, \ldots, \beta_{6q})$.
The 18 symmetry conditions are 
\[
\underbrace{\begin{bmatrix}
\blue{\Gamma H_1} &  \blue{- \Gamma H_0}\\
\green{\Gamma H_2} & &  \green{-\Gamma H_0}\\
\brown{\Gamma H_3} & & & \brown{-\Gamma H_0}\\
& \red{\Gamma H_2} & \red{-\Gamma H_1}\\
& \magenta{\Gamma H_3} & & \magenta{-\Gamma H_1} \\
& & \orange{\Gamma H_3} & \orange{-\Gamma H_2}
\end{bmatrix}}_{A(Z)} \begin{bmatrix}
v_{:,0}\\v_{:,1}\\v_{:,2}\\v_{:,3}
\end{bmatrix} = 0,
\]
where the colors in the block rows of the coefficient matrix $A(Z)$ correspond to the entries in \eqref{eq:exampleiota} on which they impose relations. 
In other words, the kernel of $A(Z)$ is the vector space $\{(v_1,\ldots,v_6) ~|~ \M(v_1,\ldots,v_6) \in \im \iota \}$ from \cref{cor:descIperp}. Hence, $\HF_{S/J(Z)}(2,1)$ is the corank of $A(Z)$. It is at least $6$, since 
\[
A(Z) \begin{bmatrix}
H_0 & H_1 & H_2 & H_3
\end{bmatrix}^\top= 0.
\]
These null vectors correspond to the $w_i'$ in the proof of \cref{cor:keratleastr}. For generic $Z \in X^r$, the corank of $A(Z)$ is exactly 6, so $\HF_{S/J(Z)}(2,1) = 6$.
\end{example}

From the discussion in \cref{ex:reg1}, we would like to conclude that, generically, $(2,1) \in \Reg(J(Z))$. For this to make sense, i.e., to apply \cref{def:reg}, we need to show that for most configurations $Z$, $J(Z)$ defines $r$ points with multiplicity one. By an argument analogous to \cref{lem:zerodim}, this happens for small enough ranks.

\begin{lemma} \label{lem:Jdefpoints}
Suppose that $r \leq mn$ (this is \cref{eqn_rank_condition} for $\ell = \infty$). There is a Zariski open, dense subset $U \subset X^r$ such that for all $Z = (\z_1, \ldots, \z_r) \in U$, $w_1, \ldots, w_r$ are $\C$-linearly independent and $V_X(J(Z)) = \{\z_1,\ldots,\z_r\}$ consists of $r$ points with multiplicity one. 
\end{lemma}
\begin{proof}
As before, let $\z_i = (\beta_i, \gamma_i), i = 1, \ldots, r$ and $w_i = (\beta_i \otimes \gamma_i)^\top \in S_{(1,1)}^\vee$. By \cite[Theorem 2.5]{chiantini2006concept}, there is an open dense subset $U' \subset X^r$ such that for $Z = (\z_1, \ldots, \z_r) \in U'$, $\Span_\C(w_1,\ldots,w_r)$ contains no points of the form $(\beta \otimes \gamma)^\top$, other than the $w_i$. We set $U = U' \cap W$, which is open and dense in $X^r$. The rest of the proof is identical to that of \cref{lem:zerodim}.
\end{proof}

Our next goal is to show that, in order to prove that $(d,e) \in \Reg(J(Z))$ for almost all $Z \in X^r$, it suffices to find one particular instance $Z^* \in W$ for which $\HF_{S/J(Z^*)}(d,e) = r$. 
\begin{proposition} \label{prop:HFusc}
Suppose $r \leq \HF_S(1,1)$ such that $W \neq \emptyset$. For fixed $(d,e) \geq (1,1)$, the Hilbert function $\HF_{S/J(Z)}(d,e)$, as a function of $Z$, is upper semicontinous on $W$. That is, for any $r^* \in \N$, 
$$ V_{r^*} = \{ Z \in W ~|~ \HF_{S/J(Z)}(d,e) > r^* \}$$
is Zariski closed in $W$. Consequently, either $V_{r^*} = W$ or $V_{r^*} \subsetneq W$ is a strict closed subvariety.
\end{proposition}
\begin{proof}
We have $\HF_{S/J(Z)}(d,e) = \HF_S(d,e) - \HF_{J(Z)}(d,e)$, and $J(Z)_{(d,e)} = S_{(d',e')} \cdot J(Z)_{(1,1)}$. Therefore, the condition $\HF_{S/J(Z)}(d,e) > r^*$ is equivalent to $\dim_\C (S_{(d',e')} \cdot J(Z)_{(1,1)}) < \HF_S(d,e) - r^*$, which can be written as the vanishing of the $(\HF_S(d,e) - r^*)$-minors of a matrix, whose entries are monomials in the coordinates of $Z$.
\end{proof}

\begin{corollary} \label{cor:pick1point}
Suppose that $r \leq mn$ and $(d,e) \geq (1,1)$. Let $U \subset X^r$ be the dense open subset from \cref{lem:Jdefpoints}. If for some element $Z^* \in W$ we have $\HF_{S/J(Z^*)}(d,e) = r$, then there is a Zariski open, dense subset $U^\circ$ of $U$ such that for all $Z \in U^\circ$, we have $(d,e) \in \Reg(J(Z))$.
\end{corollary}
\begin{proof}
By \cref{cor:keratleastr}, we know $V_{r-1} = W$. Since $Z^* \in W \setminus V_r$, \cref{prop:HFusc} implies that $V_r \subset W$ is a strict subvariety. We set $U^\circ = U \setminus V_r$. Clearly, if $Z = (\z_1, \ldots, \z_r) \in U^\circ$, then $J(Z)$ is such that $V_X(J(Z)) = \{\z_1, \ldots, \z_r \}$, where these points occur with multiplicity one (this uses $U^\circ \subset U$), and $\HF_{S/J(Z)} = r$ (since $U^\circ \subset V_{r-1} \setminus V_r$). 
\end{proof}
This result will be particularly useful for proving \cref{thm:conjcases} below. First, we investigate which combinations of $m,n,d,e,r$ are possible to have $\HF_{S/J(Z)}(d,e) = r$ for generic points $Z$. 

For $(d,e) \geq (1,1)$ and $m,n \in \N_0$ we define
 \[
 \Rc(m,n,(d,e)) = \frac{\HF_S(1,1) \HF_S(d',e') - \HF_S(d,e)}{\HF_S(d',e')-1}.
 \]
If $(d,e) = (1,1)$, we set $\Rc(m,n,(d,e)) = \infty$. 
For given $m,n$ and $(d,e) \geq (1,1)$, the following result shows that $\Rc(m,n,(d,e))$ bounds the rank $r$ for which we could possibly have $\HF_{S/J(Z)}(d,e) = r$ for $Z \in W \subset X^r$.

\begin{theorem} \label{thm:boundr}
Let $(d,e) \geq (1,1)$ and $Z \in W \subset X^r = (\PP^m \times \PP^n)^r$ with 
\begin{equation} \label{eq:badregion}
\Rc(m,n,(d,e))  < r \leq mn.
\end{equation}
We have $\HF_{S/J(Z)}(d,e) > r$.  In particular, for $r$ in the range \eqref{eq:badregion} and $Z \in U$, where $U \subset W$ is the open subset from \cref{lem:Jdefpoints}, we have $(d,e) \notin \Reg(J(Z))$.
\end{theorem}
\begin{proof}
Since $\HF_{S/J(Z)}(d,e) = \HF_S(d,e) - \HF_{J(Z)}(d,e) \geq \HF_S(d,e) - \HF_S(d',e') \HF_{J(Z)}(1,1)$
and $\HF_{J(Z)}(1,1) = \HF_S(1,1) - r$, we have 
\begin{equation} \label{eq:ineqHFSJ}
\HF_{S/J(Z)}(d,e) \geq \HF_S(d,e) + r \HF_S(d',e') - \HF_S(1,1) \HF_S(d',e').
\end{equation}
Solving $\HF_S(d,e) + r \HF_S(d',e') - \HF_S(1,1) \HF_S(d',e') > r$ yields the first inequality in \eqref{eq:badregion}.
\end{proof}

Note that if $(d,e) = (1,1)$, the range \eqref{eq:badregion} is empty. This agrees with the fact that $(1,1) \in \Reg(J(Z))$ for all $Z \in W$. From \cref{cor:keratleastr} we know that $\HF_{S/J(Z)}(d,e) \geq r$ for $(d,e) \geq (1,1)$ and $Z \in W$. Combining this with a dimension argument as in the proof of \cref{thm:boundr}, we see that 
\begin{equation} \label{eq:max}
 \HF_{S/J(Z)}(d,e) \geq \max \{ r~,~ \HF_S(d,e) + r \HF_S(d',e') - \HF_S(1,1)\HF_S(d',e') \}.
 \end{equation}
For $e' = 0$ or $d' = 0$, we observe experimentally that equality holds for generic configurations $Z = (\z_1, \ldots, \z_r) \in X^r$. In this case, it suffices to check for which $r$ the maximum equals $r$. This leads us to the following conjecture.

\begin{conjecture} \label{conj:reg}
Let $d \geq 1$ and let $m,n,r$ be such that 
\begin{equation} \label{eq:goodregion}
r \leq \min \left \{  \Rc(m,n,(d,1)) ~ , ~mn \right  \}.
\end{equation}
There is a Zariski open, dense subset $U^\circ \subset U \subset X^r = (\PP^m \times \PP^n)^r$, were $U$ is the open subset from \cref{lem:Jdefpoints}, such that for all $Z \in U^\circ$, $(d,1) \in \Reg(J(Z))$. 
\end{conjecture}

\begin{theorem}\label{thm:conjcases}
\Cref{conj:reg} holds in the following cases: 
\begin{enumerate}
\item $m \in \N_0, ~n \in \N_0$ and $d = 1$,
\item $m \in \{1,2\}, ~ n \in \N_0$ and $d = 2$,
\item $2 \leq m+1, n+1 \leq 50$ and $d = 2$,
\item $2 \leq m+1, n+1 \leq 9$ and $3 \le d \le 5$,
\item $2 \leq m+1,n+1 \leq 6$ and $6 \le d \le 10$.
\end{enumerate}
\end{theorem}
\begin{proof}
(1) In this case, clearly $U^\circ = U$.

(2) Following \cref{ex:reg1}, we can compute $\HF_{S/J(Z)}(2,1)$ as the corank of the stacked matrix \(A(Z) = [\begin{smallmatrix}
G_1(Z)^\top & \cdots & G_m(Z)^\top
\end{smallmatrix}]^\top\), where 
\[
G_k(Z) = 
\begin{bmatrix}
0_{(n+1)\times (k-1)r} & \Gamma H_{k} & -\Gamma H_{k-1} \\
0_{(n+1)\times (k-1)r} & \Gamma H_{k+1} & & -\Gamma H_{k-1} \\
\vdots & \vdots & & & \ddots \\
0_{(n+1)\times (k-1)r} & \Gamma H_{m} & & & & -\Gamma H_{k-1} \\ 
\end{bmatrix},
\]
the matrix $\Gamma$ contains homogeneous coordinates of the points $\gamma_i$ in its columns, and $H_q$ is the diagonal matrix $\diag(\beta_{1q}, \ldots, \beta_{rq})$.

By \cref{cor:pick1point}, for each $r$ satisfying \eqref{eq:goodregion} we must find one $Z^* \in W$ such that $\HF_{S/J(Z^*)}(2,1) = \text{corank}(A(Z^*)) = r$. 
In fact, since $A(Z) ~ (H_0 ~ \cdots ~ H_m)^\top = 0$, we have $\text{corank}(A(Z)) \geq r$ for any $Z \in X^r$, and by upper semicontinuity of corank, if $\text{corank}(A(Z^*)) = r$ for some $Z^* \in X^r$, then $\text{corank}(A(Z)) = r$ for all $Z$ in a dense, Zariski open subset of $X^r$.
For $m = 1$, \eqref{eq:goodregion} entails $r \leq n +1$, which means that the matrix $\Gamma$ has more rows than columns. On a dense, Zariski open subset of $X^r$, $\Gamma$ has rank $r$ and, since $(\beta_{i0},\beta_{i1}) \neq (0,0)$, the matrix $A = [\Gamma H_1 ~ - \Gamma H_0]$ has rank at least $r$. The proof for $m=2$ is more technical. Since we could not generalize it for higher $m$, we have deferred it to \cref{app:proofs}.

(3)--(5) These cases consists of a computer-assisted proof.\footnote{The code and certificates can be obtained at \url{https://gitlab.kuleuven.be/u0072863/homogeneous-normal-form-cpd}.}

For proving (3), we generated random $Z^*$ and confirmed that $A(Z^*)$ has corank $r$ as follows. We generate random $\beta_i \in \mathbb{Z}^{m+1}$ and $\gamma_i \in \mathbb{Z}^{n+1}$. Then $A(Z)$ is a matrix over $\mathbb{Z}$. We can upper bound its corank by computing the corank, via a reduction to row echelon form using Gaussian elimination, over the finite field $\mathbb{Z}_p$ for some prime number $p$. We used $p=8191$. The corank of $A(Z^*)$ over the finite field $\mathbb{Z}_p$ is then an upper bound of the corank of $A(Z^*)$ over $\mathbb{Z}$ (and so also its algebraic closure). We implemented this approach in C++ and certified for all listed cases that there exists a $Z^*$ (with integer coefficients) such that $A(Z)$ has corank $r$.   

For proving (4) and (5), we apply \cref{cor:pick1point} to verify \cref{conj:reg} by computing the Hilbert function in Macaulay2 \cite{M2} for all listed cases as follows:
{\small\begin{verbatim}
F = Z/8191;
HF = (x,y) -> binomial(m+x,x)*binomial(n+y,y);
degs = apply(m+1,i->{1,0}) | apply(n+1,i->{0,1});
S = F[x_0..x_m,y_0..y_n, Degrees => degs];
X = matrix apply(m+1,i->apply(r,j->random(F)));
Y = matrix apply(n+1,i->apply(r,j->random(F)));
N = transpose matrix apply(r,i->(X_i**Y_i));
I = ideal (basis({1,1},S)*(gens ker N));
rk = rank N; -- check that this is r, so Z* is in W
hf = hilbertFunction({d,1},I); -- check that this equals r
\end{verbatim}}
where \texttt{m}, \texttt{n}, \texttt{d}, and \texttt{r} are respectively $m$, $n$, $d$, and $r$ from the theorem.
\end{proof}

Note that the role of $m$ and $n$ in \cref{conj:reg} can be interchanged, implying the analogous statement that for $e \geq 1$ and generic $Z \in W \subset X^r = (\PP^m \times \PP^n)^r$ with 
\(
r \leq \min \left \{ \Rc(m,n,(1,e)), mn \right \},
\)
we have $(1,e) \in \Reg(J(Z))$. 

Let $\A \in \C^{\ell + 1} \otimes \C^{m + 1} \otimes \C^{n+1}$ be a general tensor whose rank satisfies \eqref{eqn_rank_condition} and let $I = \ideal{ \ker \A_{(1)} } = J(Z)$, where $Z = ((\beta_i, \gamma_i))_{i=1, \ldots, r}$. A consequence of \cref{conj:reg} would be that $(d,1) \in \Reg(I)$, for any $d$ such that $\Rc(m,n,(d,1)) \geq r$ and $(1,e) \in \Reg(I)$ for any $e$ such that $\Rc(m,n,(1,e)) \geq r$. 

\begin{example}[\Cref{ex:bigleap}, continued]
In \cref{ex:bigleap}, we had $m = 6$ and $n = 2$. We have $\Rc(6,2,(2,1)) = 21/2 < r = 12$, which explains why $(2,1) \notin \Reg(I)$. Also, $\Rc(6,2,(3,1)) = 112/9 \geq 12$ implies, since $\A$ is generic, that $(3,1) \in \Reg(I)$. Similarly, the smallest $e > 1$ for which $\Rc(6,2,(1,e)) \geq 12$ is $e = 5$, so that $(1,5) \in \Reg(I)$, but $(1,e) \notin \Reg(I)$ for $1 < e < 5$. These computations are confirmed by \cref{tab:hfex4}.
\end{example}

One can check that
\begin{align}
\Rc(m,n,(d,1)) 
\label{eqn_Rc_expression} &= \frac{\binom{m+d-1}{d-1}}{\binom{m+d-1}{d-1} - 1} \left( (m+1)(n+1) - \frac{n+1}{d}(m+d) \right).
\end{align}
The first factor is always greater than $1$ and the second factor is at least $mn$ if $d \ge n+1$. If \cref{conj:reg} holds, then for all tensors $\A$ of rank $r$ in the range \eqref{eqn_rank_condition}, we have $(n+1,1) \in \Reg(I)$. In that case, \algoname{} can take $(d,e) = (n+1,1)$ to treat all identifiable rank-$r$ tensors in the \emph{unbalanced} regime. That is, when $\ell > mn$. For such shapes, \cite{BCO2013} proved that generic $r$-identifiability holds up to $r \le mn$.

\Cref{conj:reg} gives a way of finding degrees of the form $(d,1) \in \Reg(I) \setminus \{(1,1)\}$.
There might exist other tuples of the form $(d,e) \in \Reg(I) \setminus \{(1,1)\}$, which could lead to a lower computational cost. For such degrees, \eqref{eq:max} may be a strict inequality, which means that other tools are needed. We leave the further exploration of effective bounds for the regularity for future research.

\section{Computational complexity} \label{sec:complexity}
One motivation for studying the regularity of the ideal $I = \langle \ker \A_{(1)} \rangle$ is to understand the computational complexity of \cref{alg:pseudoalg}. 

Assume that we are given an $r$-identifiable tensor $\A$ in $\C^{\ell+1} \otimes \C^{m+1}\otimes \C^{n+1}$, where $\ell\ge m\ge n$. We additionally assume that $\ell+1 \le (m+1)(n+1)$.\footnote{A tensor $\A$ not satisfying this constraint on $\ell$ can always be represented in new bases by a coordinate array $\mathcal{B}$ in a \textit{concise} \cite{Landsberg2012,BCS1997} tensor space $\C^{\ell'+1} \otimes \C^{m'+1}\otimes \C^{n'+1}$ with $\ell' \le \ell$, $m' \le m$, and $n'\le n$. After permutation of the factors, a concise tensor space $\C^{k_1}\otimes\C^{k_2}\otimes\C^{k_3}$ always satisfies $k_2k_3 \ge k_1 \ge k_2 \ge k_3$; see, e.g., \cite{CK2011}. In practice, $\mathcal{B}$ can be obtained as the core tensor of the (sequentially) truncated higher-order singular value decomposition \cite{LMV2000,VVM2012}.} 
The size of the input will be denoted by $M = (\ell+1)(m+1)(n+1)$. Because of the constraint $(m+1)(n+1) \ge \ell+1 \ge m+1 \ge n+1$, we have
\[
 M^{\frac{1}{4}} \le m+1 \le M^{\frac{1}{2}} \text{ and } 2 \le n+1 \le M^{\frac{1}{3}}.
\]
Since \algoname{} applies only if the rank $r$ satisfies \cref{eqn_rank_condition}, we have $r = \mathcal{O}(M^\frac{5}{6})$. 

Without going into details, it can be verified that a crude upper bound for the computational complexity of the steps in \cref{alg:pseudoalg}, excluding step 4, is $\mathcal{O}(M^3)$. 

The key contribution to the time complexity originates from \cref{alg:mhnfalg}. 
Its complexity is ultimately determined by the choice of $(d,e)$ in line \ref{line_alg_degree} of the algorithm. 
We analyze what happens when the degree is selected as $(d,1)$. 
In this case, bounding the size of the matrix $R_I(d,e)$ in \cref{alg:mhnfalg} is critical. Assuming \cref{conj:reg} holds, we have 
\begin{equation} \label{eq:ineqcomp}
 \left ( \frac{3}{2} \right ) ^{d+1} \le \HF_S(d,1) = \binom{m+d}{d} (n+1) \le M^{\frac{5}{6}d} M^{\frac{1}{3}}.
\end{equation}
The upper bound follows from $\binom{m+d}{d} \le (m+d)^d \le (2(m+1))^d \le ((m+1)(n+1))^d$.
For establishing the lower bound, note that $\left(1 + \frac{m}{d}\right)^d \le \binom{m+d}{d}$ and use the facts that $d \le n+1 \le m+1$ (see the discussion below \eqref{eqn_Rc_expression}), so that $d\ge2$ implies that $\frac{m}{d} \ge \frac{1}{2}$. Computing a basis for the left null space of $R_{I}(d,1)$ requires at most $\mathcal{O}( (\HF_S(d,1))^3 )$ and at least $r \cdot \HF_S(d,1)$ operations, as we know $(d,1)\in\Reg(I)$ so that the dimension of the left null space is $r$. Consequently, line \ref{line_null_space} has a time complexity that is at least exponential in $d$.
 
\begin{proposition} \label{prop_exponential}
 Consider the concise tensor space $\C^{M^{\frac{1}{2}}}\otimes\C^{M^{\frac{1}{4}}}\otimes\C^{M^{\frac{1}{4}}}$. If \cref{conj:reg} holds, then for a generic tensor of rank $r=M^{\frac{1}{2}} - 2M^{\frac{1}{4}} - 1$ in this space, the asymptotic time complexity of \cref{alg:pseudoalg} is at least exponential in the input size $M = M^{\frac{1}{2}} M^{\frac{1}{4}} M^{\frac{1}{4}}$, if the degrees are restricted to $(d,1)$ or $(1,e)$.
\end{proposition} 
\begin{proof}
 It follows from \cref{eqn_Rc_expression} and \cref{thm:boundr} that for sufficiently large $M$, the degree $d$ should be at least $\frac{1}{2} M^{\frac{1}{4}}$. Combining this with the above discussion about the size of $R_I(d,1)$ concludes the proof.
\end{proof} 
 
One might conclude from this result that \algoname{} is not an effective algorithm for tensor rank decomposition. However, \cite{HL2013} proved that computing the rank of a tensor over $\C$ is an NP-hard problem. Hence, a polynomial-time algorithm \textit{that applies to all inputs} for this problem is not anticipated.
A \textit{typical} instance of an NP-hard problem can often be solved more efficiently than the worst-case instance. \Cref{prop_polynomial} is in this spirit.

\begin{proof}[Proof of \cref{prop_polynomial}]
 It follows from the upper bound in \eqref{eq:ineqcomp} that the time complexity of \cref{alg:pseudoalg} is $\mathcal{O}( (\HF_S(d,1))^3 ) = \mathcal{O}( M^{\frac{5}{2}d + 1} )$. 
 
We determine the smallest degree $d$ under \cref{conj:reg} such that $(d,1) \in \Reg(I)$, where $I=\langle \ker \A_{(1)} \rangle$. From \cref{thm:boundr} we know that a necessary condition for this is that $\Rc(m,n,(d,1)) \ge r = \phi mn$. 
The last inequality is implied by 
$
  (m+1)(n+1) - \frac{n+1}{d}(m+d) \ge \phi mn
$
because of \cref{eqn_Rc_expression}. This is equivalent to 
\[
 (1 - \phi) mn + m + n + 1 \ge 
 \frac{1}{d}(m+d)(n+1) 
 = \frac{1}{d} m (n+1) + n+1.
\]
Hence, $\Rc(m,n,(d,1)) \ge \phi mn$ is implied by
\[
 d \ge \frac{1}{1 - \phi} > (n+1) \frac{1}{(1 - \phi)n + 1} = \frac{m(n+1)}{(1-\phi)mn + m}.
\]
In other words, provided \cref{conj:reg} holds, it suffices to take $d = \lceil \frac{1}{1-\phi} \rceil$ to be able to cover all ranks up $\phi mn$. As $\phi \ne 1$ is a constant, this proves the result.
\end{proof}

We can observe that the formula $d \ge \frac{1}{1-\phi}$ is asymptotically sharp in the sense that as $\phi\to1$ the exponent in $\mathcal{O}(M^\star)$ needs to blow up to $\infty$ and no polynomial in the input size can control the growth. Indeed, \cref{prop_exponential} shows that exactly when $\phi=1$ there exist cases that require at least an exponential growth. Note that by the discussion below \eqref{eqn_Rc_expression}, assuming \cref{conj:reg} we can use $d = n+1$ for any $\phi \in [0,1]$. This gives the universal, exponential complexity bound $\mathcal{O}(M^{\frac{5}{2}n + \frac{7}{2}})$.

\section{Numerical experiments} \label{sec:experiments}

We present several numerical results demonstrating the effectiveness of \algoname. All experiments were performed on KU Leuven/UHasselt's Tier-2 Genius cluster of the Vlaams Supercomputer Centrum (VSC). 
Specifically, the supercomputer's scheduling software allocated standard \textit{skylake} nodes containing two Xeon\circledR{} Gold 6140 CPUs (18 physical cores, 2.3GHz clock speed, 24.75MB L3 cache) with 192GB of main memory, as well as standard \textit{cascadelake} nodes which are equipped with two Xeon\circledR{} Gold 6240 CPUs (18 physical cores, 2.6GHz clock speed, 24.75MB L3 cache) with 192GB of main memory for our experiments. In all experiments, we allowed all algorithms to use up to 18 physical cores.

The proposed algorithm was implemented in Julia v1.4.0, relying on the non-base packages Arpack.jl, DynamicPolynomials.jl, GenericSchur.jl, and MultivariatePolynomials.jl.\footnote{Additional packages are used to support our experimental setup, but these are not required for the main algorithm.} Our implementation follows the pseudocode in \cref{alg:pseudoalg,alg:mhnfalg} and the detailed discussion in \cref{sec_sub_hnfalgorithm} quite closely. The Julia code of \algoname, including driver routines to reproduce \algoname's experimental results can be found at \url{https://gitlab.kuleuven.be/u0072863/homogeneous-normal-form-cpd}.

In the experiments below, random rank-$r$ tensors $\A = \sum_{i=1}^r \alpha_i \otimes \beta_i \otimes \gamma_i$ are generated by randomly sampling the elements of $\alpha_i \in \R^{\ell+1}$, $\beta_i \in \R^{m+1}$, and $\gamma_i \in \R^{n+1}$ identically and independently distributed (i.i.d.) from a standard normal distribution.

\subsection{Implementation details}
The algorithm is implemented for real and complex input tensors. In the former case, all computations are performed over the reals with the exception of the computation of the left null space of $R_I(d,e)$. In the real case, the algorithm continues with the real part of the output of this step.

At the start of the algorithm, we compress the $(\ell+1) \times (m+1) \times (n+1)$ input tensor, which is to be decomposed into $r$ rank-$1$ terms, to a $\min\{\ell+1,r\} \times \min\{m+1,r\} \times \min\{n+1,r\}$ tensor. For this we apply ST-HOSVD compression \cite{VVM2012} with truncation rank $(\min\{\ell+1,r\}, \min\{m+1,r\}, \min\{n+1,r\})$. Most of the experiments below are chosen so that this step performs no computations.

The kernel of $\A_{(1)}$ in \cref{alg:pseudoalg} is computed by an SVD. 
The linear system at the end of \cref{alg:pseudoalg} is solved by computing the Khatri--Rao product $K = [\beta_i \otimes \gamma_i]_{i=1}^r$ and then solving the overdetermined linear system $K A = \A_{(1)}^\top$ for $A$. The rows of $A$ then correspond to the $\alpha_i$'s.

In \cref{alg:mhnfalg}, the degree $(d,e)$ is determined automatically by assuming \cref{conj:reg} is true and selecting a valid $(d,1)$ or $(1,e)$ based on a heuristic that takes into account the estimated computational cost and numerical considerations.\footnote{ For more details, see the function \texttt{minimumMultiDegree} in NormalFormCPD.jl.}

The key computational bottleneck is the computation of the left nullspace of $R_I(d,e)$ in line \ref{line_null_space} of \cref{alg:mhnfalg}. When the number of entries of $R_I(d,e)$ is smaller than 10,000, we use the standard SVD-based kernel computation. In larger cases, for efficiency, we propose to employ Arpack.jl's \texttt{eigs} function to extract the left null space from the Hermitian Gram matrix $G = R_I(d,e) (R_I(d,e))^H$, where $^H$ denotes the Hermitian transpose.
The \texttt{eigs} function (with parameters \texttt{tol = 1e-6} and \texttt{maxiter = 25}) extracts the $r$ eigenvalues of smallest modulus and their corresponding eigenvectors. 
Note that Arpack.jl employs a Bunch--Kaufman factorization \cite{BK1977} of the positive semidefinite input matrix $G$ to perform its Lanczos iterations. Using \texttt{eigs} was up to $75\%$ faster than the SVD-based approach for large problems.

The final step of \cref{alg:mhnfalg} is implemented as discussed at the end of \cref{sec_sub_hnfalgorithm}. The kernel is computed with an SVD and $\gamma_i$ is taken as the left singular vector corresponding to the smallest singular value. After this step, we find $(\beta_i, \gamma_i)$.  
Observe that $f_j(\beta_i, \gamma_i)$ should vanish exactly for all $j = 1, \ldots, s$.
We propose to refine the approximate roots $(\beta_i, \gamma_i)$ by applying three Newton iterations on $f_1 = \cdots = f_s = 0$. This adds a minor computational overhead of $r(m+1)(n+1)(m+n)^2$. 

\subsection{Impact on the accuracy of computational optimizations}\label{sec_sub_accuracy}
We compare and evaluate the impact on numerical accuracy of four variants of \cref{alg:mhnfalg}. Two options are considered in each combination: 
\begin{enumerate}
\item[(i)] refining the approximate roots $(\beta_i,\gamma_i)$ with $3$ Newton steps (\verb|+newton|) or not, and 
\item[(ii)] computing the left nullspace of $R_I(d,e)$ with Arpack.jl's iterative \verb|eigs| method applied to the Gram matrix or using the standard \verb|svd|. 
\end{enumerate}
We append \verb|_eigs|, \verb|_eigs+newton|, \verb|_svd|, and \verb|_svd+newton| to \algoname{} to distinguish the variants.

The experimental setup is as follows.
Random rank-$r$ tensors are generated as explained at the start of this section, sampling one rank-$r$ tensor for each combination $(\ell+1,m+1,n+1,r)$ with $25 \ge m + 1 \ge n+1 \ge 2$ and $\ell+1=r=\min\{\mathcal{R}(m,n,(2,1)),mn\}$ where $\mathcal{R}$ is as in \cref{eqn_Rc_expression}.
For each variant of \algoname{}, the \textit{relative backward error}
\(
 \left\| \A - \sum_{i=1}^r \alpha_i' \otimes \beta_i' \otimes \gamma_i' \right\|_F  \| \A \|_F^{-1},
\) and execution time is recorded,
where $(\alpha_i', \beta_i', \gamma_i')$ are the estimates obtained by the algorithm. 

\begin{figure}
 \begin{minipage}{.49\textwidth}\centering\small 
  \includegraphics[width=\textwidth]{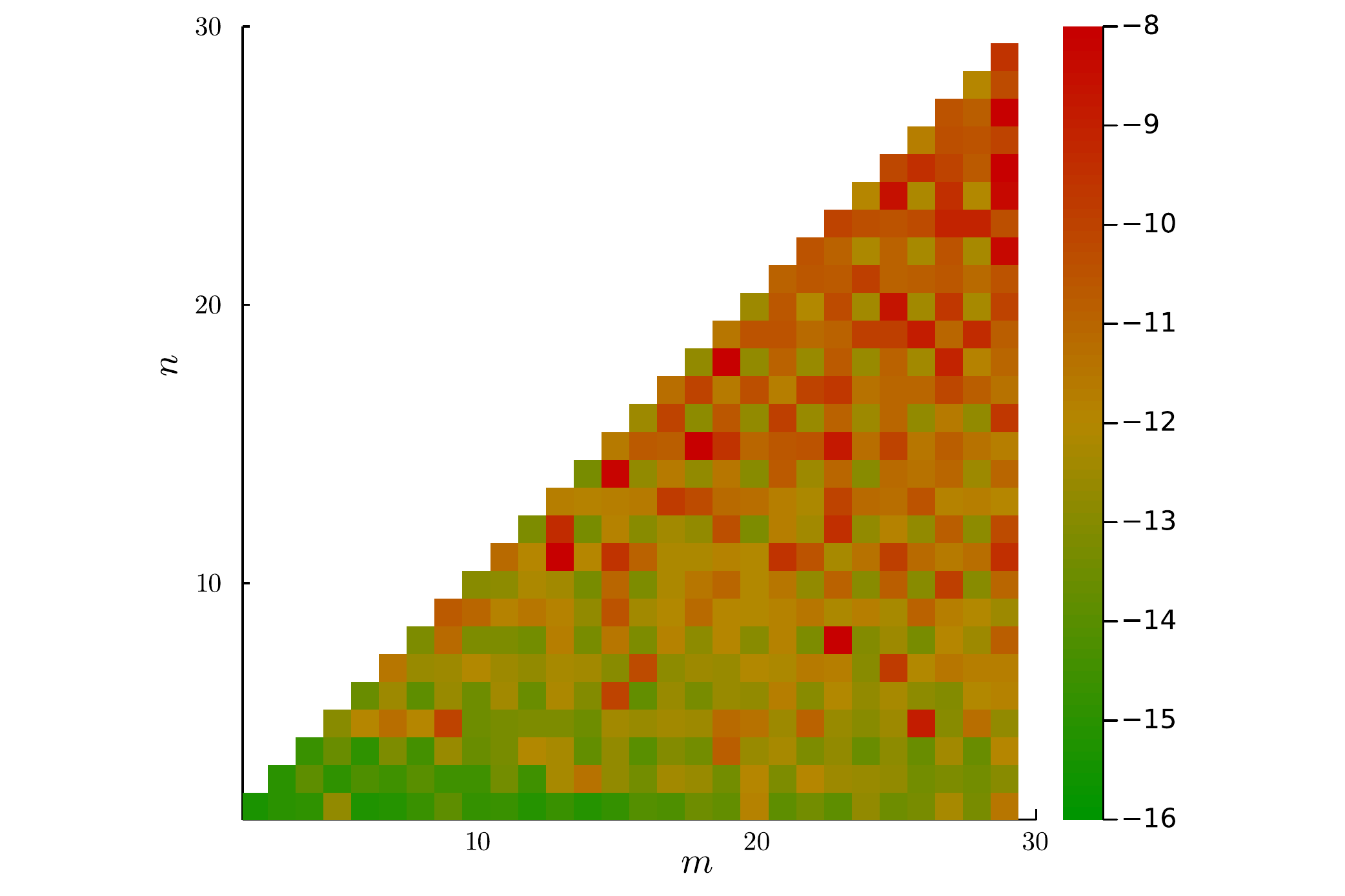}\\
  (a) \verb|cpd_hnf_eigs|
  
  \includegraphics[width=\textwidth]{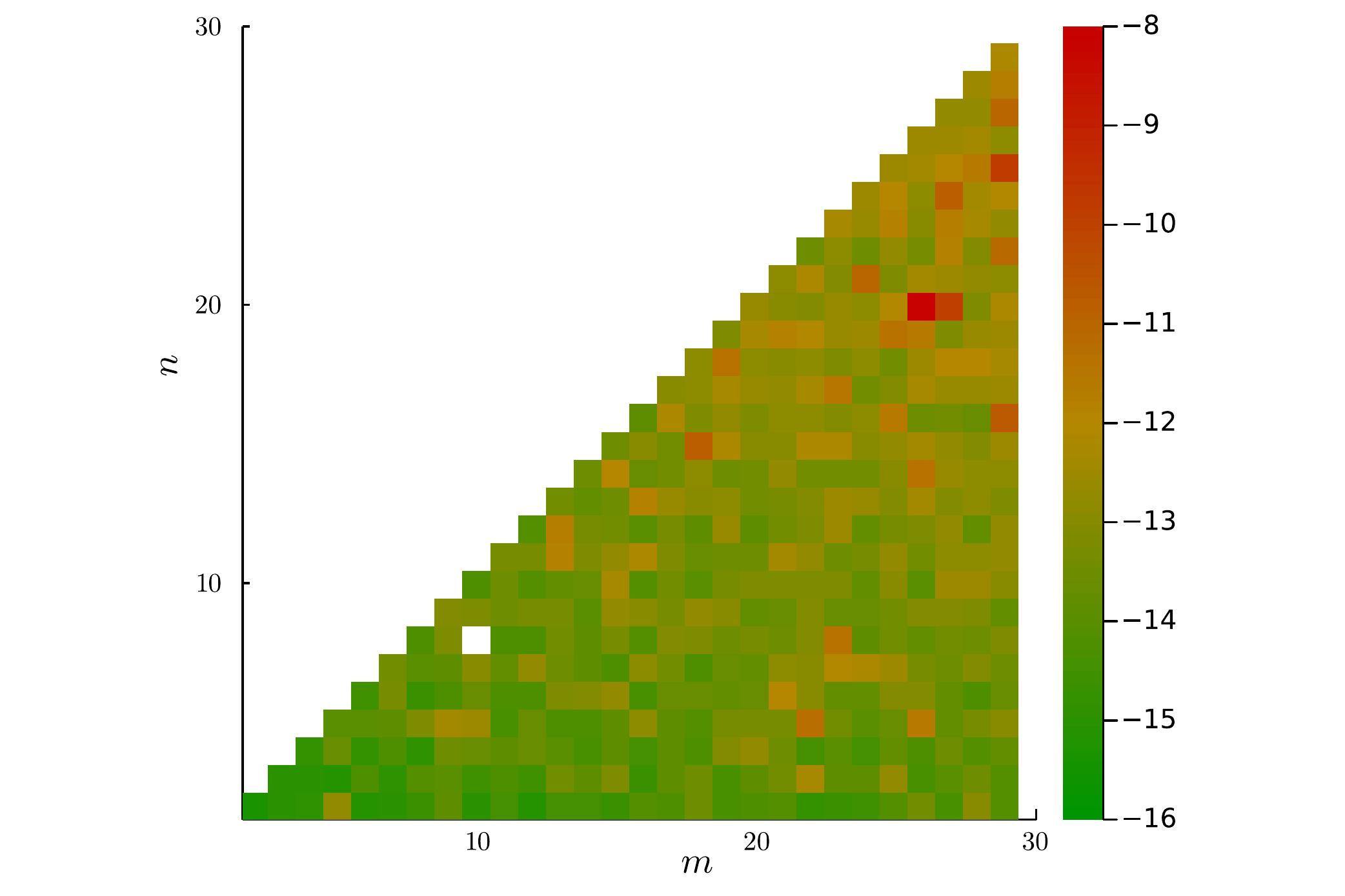}
  (c) \verb|cpd_hnf_svd|
 \end{minipage}
\hfill 
 \begin{minipage}{.49\textwidth}\centering\small 
  \includegraphics[width=\textwidth]{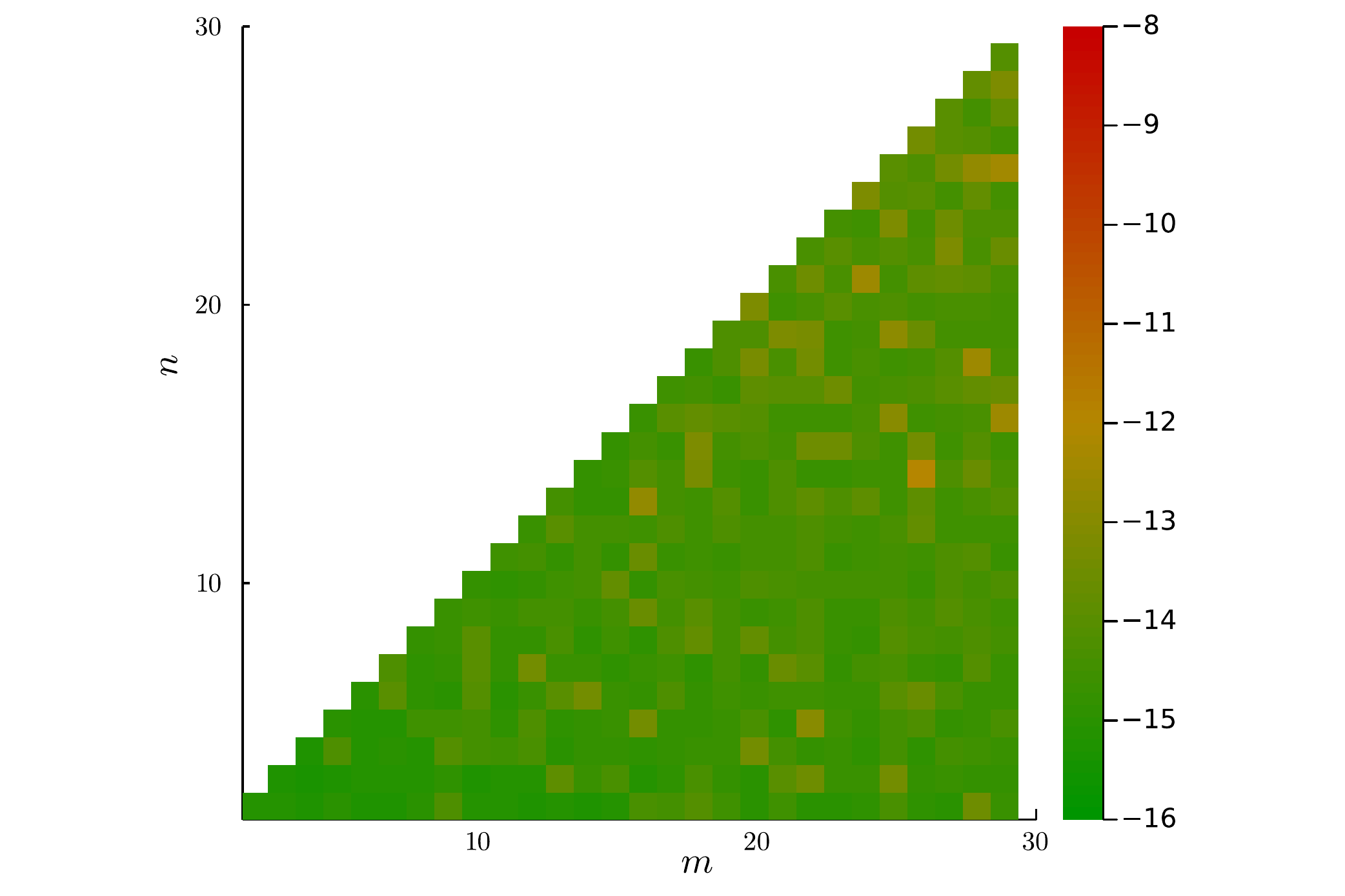}\\
  (b) \verb|cpd_hnf_eigs+newton|
  
  \includegraphics[width=\textwidth]{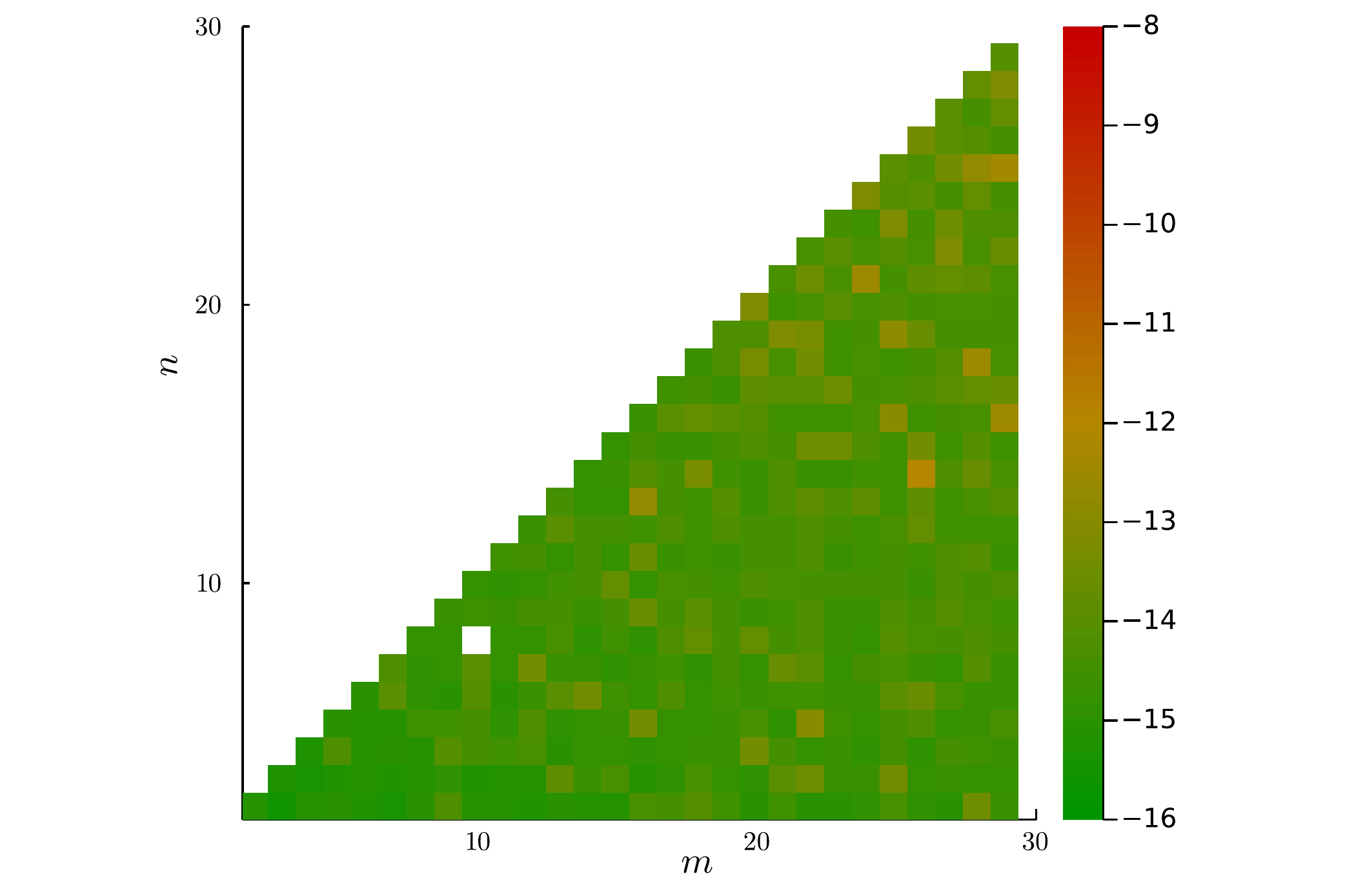}
  (d) \verb|cpd_hnf_svd+newton|
 \end{minipage}
 
 \caption{A comparison of the $\log_{10}$ of the relative backward error of the proposed \algoname{} in the four variants discussed in \cref{sec_sub_accuracy} on random rank-$r$ tensors in $\R^{\ell+1} \otimes \R^{m+1} \otimes \R^{n+1}$ with $\ell\ge m\ge n$. The largest dimension and the rank satisfy $\ell+1 = r = \min\{\mathcal{R}(m,n,(2,1)), mn\}$, where the function $\mathcal{R}$ is as in \cref{eqn_Rc_expression}. The color scale is the same in all plots.}
 \label{fig_variants}
\end{figure}

The experiment was performed on \textit{cascadelake} nodes and the results are displayed in \cref{fig_variants}. It is visually evident in the right panels that the Newton refinement improves the overall relative backward error for both \verb|eigs| and \verb|svd| by several orders of magnitude and brings both to seemingly the same level of about $10^{-16}$ to $10^{-12}$. On the other hand, there is a marked difference in accuracy between \verb|eigs| and \verb|svd| \emph{before} Newton refinement, the \verb|svd| being more accurate by about 1--3 orders of magnitude. It is interesting to compare these results with the leftmost panel of \cref{fig_accuracy}(b), which shows the corresponding results on the same tensors for a conceptually similar state-of-the-art method from the literature. The method also requires a kernel computation that, in the displayed graph, is performed with a singular value decomposition. This means that the methods in \cref{fig_variants}(c) and the left panel of \cref{fig_accuracy}(b) can be directly compared. It can be seen that the proposed \verb|cpd_hnf_svd| method, even without Newton refinement, can be up to $3$ orders of magnitude more accurate (in relative backward error). The differences in accuracy are thus not exclusively attributable to the Newton refinement.

We did not include a visualization of the timings because they visually look very similar. By adding the Newton refinement, the total execution time over all cases only modestly increased: $3.4\%$ for \verb|eigs| and $3.8\%$ for \verb|svd|. The relative increase was larger for the small cases while it was close to $0$--$5$\% for the large cases ($n+1 \ge 20$). Given that the accuracy can be improved by orders of magnitude, we conclude from these experiments that a few steps of Newton refinement of the approximate roots $(\beta_i,\gamma_i)$ is highly recommended.

After deciding that Newton refinement is our default choice and observing in the right panels of \cref{fig_variants} there is no significant difference in accuracy between the \verb|eigs| and \verb|svd|, we investigate their relative computational performance. In terms of total execution time over all cases, the \verb|eigs| variant is $29.1$\% faster. For the largest case this increases to a $40.9$\% reduction in execution time.

From the above experiments, we conclude that the variant with an iterative eigensolver and Newton refinement is an appropriate default choice. In the remainder of the experiments, \verb|cpd_hnf| refers to \verb|cpd_hnf_eigs+newton|, the variant using Arpack's iterative eigenvalue method applied to the Gram matrix combined with $3$ Newton steps to refine the approximate roots.

\subsection{Comparison with the state of the art}
In the second experiment, we compare \algoname{} with the current state of the art in direct numerical methods for tensor rank decomposition in terms of accuracy and computational performance. The algorithm developed by \cite{domanov2017canonical,domanov2014canonical} is an advanced direct numerical method that shares several high-level characteristics with \algoname{}:
\begin{enumerate}
 \item[(i)] Both methods assume the target rank $r$ is supplied to the decomposition algorithm.  They operate in a similar regime of ranks and are able to treat the full range of generically $r$-identifiable tensors in unbalanced tensor spaces. The method by Domanov and De Lathauwer can even deal with some ranks $r > \ell+1 \ge m+1 \ge n+1$, while our algorithm cannot.
 
 \item[(ii)] Like our method, Domanov and De Lathauwer rely on a simultaneous diagonalization procedure to extract one or two of the \emph{factor matrices}.
 
 \item[(iii)] To deal with high ranks, both methods rely on the construction of an auxiliary matrix whose size is parameterized by an integer. By increasing this integer, the range of ranks that is covered by the algorithms is broadened at the cost of a substantially increased computational complexity. In both algorithms, the asymptotic computational complexity is determined by the cost of computing the kernel of this auxiliary matrix. Contrary to \cite{domanov2017canonical,domanov2014canonical}, we are able to give a precise connection between this integer and the range of ranks we can cover (subject to \cref{conj:reg}).
 
 \item[(iv)] When $\ell+1 \ge m+1 \ge r$, both approaches can be considered pencil-based algorithms. For this reason, our comparison will focus on the case where $\ell+1 \ge r \ge m+1 \ge n+1$. 
\end{enumerate}

A Matlab implementation of Domanov and De Lathauwer's algorithm \cite{domanov2017canonical} is available as part of the Tensorlab+ repository \cite{tensorlabplus2022} as \texttt{cpd3\_gevd} in the \texttt{domanov2017laa} directory. We refer to it as \texttt{cpd\_ddl} henceforth. \texttt{cpd\_ddl} has an accurate and a fast option for computing the kernel of the auxiliary matrix. We found in our experiments that the dimension of this kernel is often misjudged by the fast implementation. Consequently, it usually increases the size of the auxiliary matrix and causes it to exceed the $16.2$GB memory limit we put on the size of that matrix. As a concrete statistic, in the configuration with $d=2$ below, the fast version of \texttt{cpd\_ddl} failed in over 75\% of the cases. Therefore, we exclusively compare our implementation with the accurate version \texttt{cpd\_ddl}. 

We generate random rank-$r$ tensors $\A$ as described at the start of this section.
We apply both \algoname{} and \texttt{cpd\_ddl} to these random tensors for all of the following configurations:\footnote{Both algorithms were applied to the same rank-$r$ tensors. To deal with the different programming languages and to limit storage demands, we generated a buffer of $10^7$ reals sampled i.i.d.~from a standard normal distribution. The $\alpha_i$, $\beta_i$, and $\gamma_i$ of the true decomposition were then generated from this buffer. This entails there is a statistical correlation between the various tensors that are decomposed. However, we judged that this does not affect the validity of our conclusions. 
For each line in the set of configurations, a different random buffer was generated.}
\begin{align*}
&40 \ge m+1 \ge n+1 \ge 2 \quad\text{for}\quad d=2, \text{ and}\\
&20 \ge m+1 \ge n+1 \ge 2 \quad\text{for}\quad d=3, \text{ and}\\
&15 \ge m+1 \ge n+1 \ge 2 \quad\text{for}\quad d=4,
\end{align*}
and in all cases we take $\ell+1 = r = \min\{\mathcal{R}(m, n, (d,1)), mn\}$. We do not provide our algorithm with the value of $d$. For each input tensor, \algoname{} determines the degree $(d,1)$ or $(1,e)$ automatically as explained in the previous subsection.
For each algorithm, we record the relative backward error, total execution time, and the size of the auxiliary matrix whose kernel is computed. As the latter is the dominant operation in both algorithms, its time and memory complexity gives a good estimate of the overall complexity.

\begin{figure}[t]
\includegraphics[width=\textwidth]{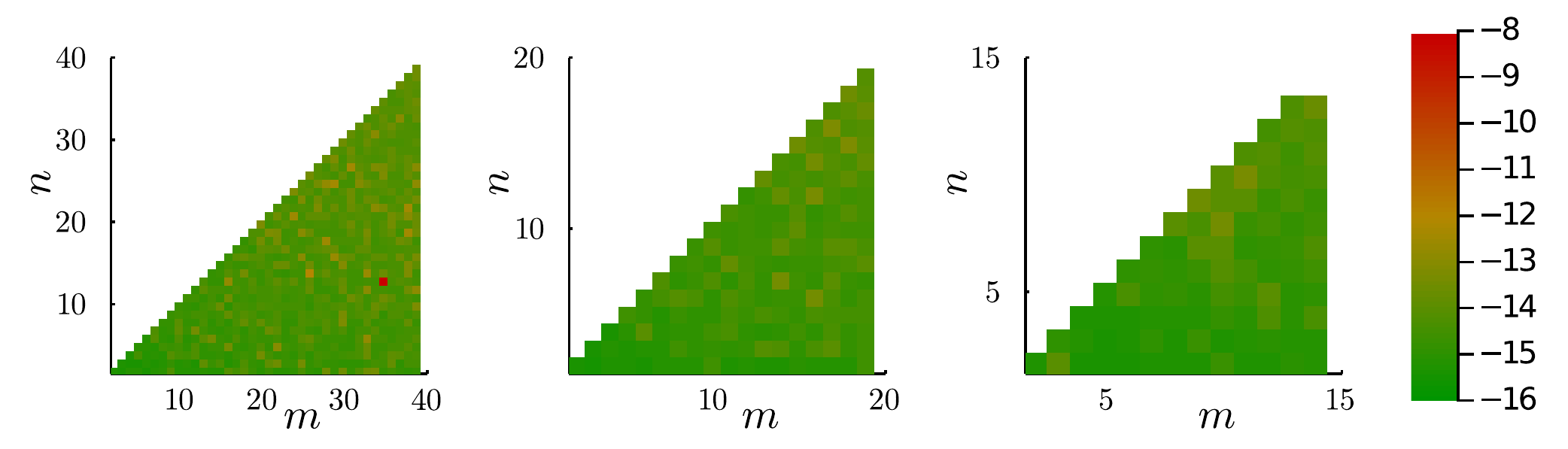}\\
\centerline{\small (a) \algoname}
\includegraphics[width=\textwidth]{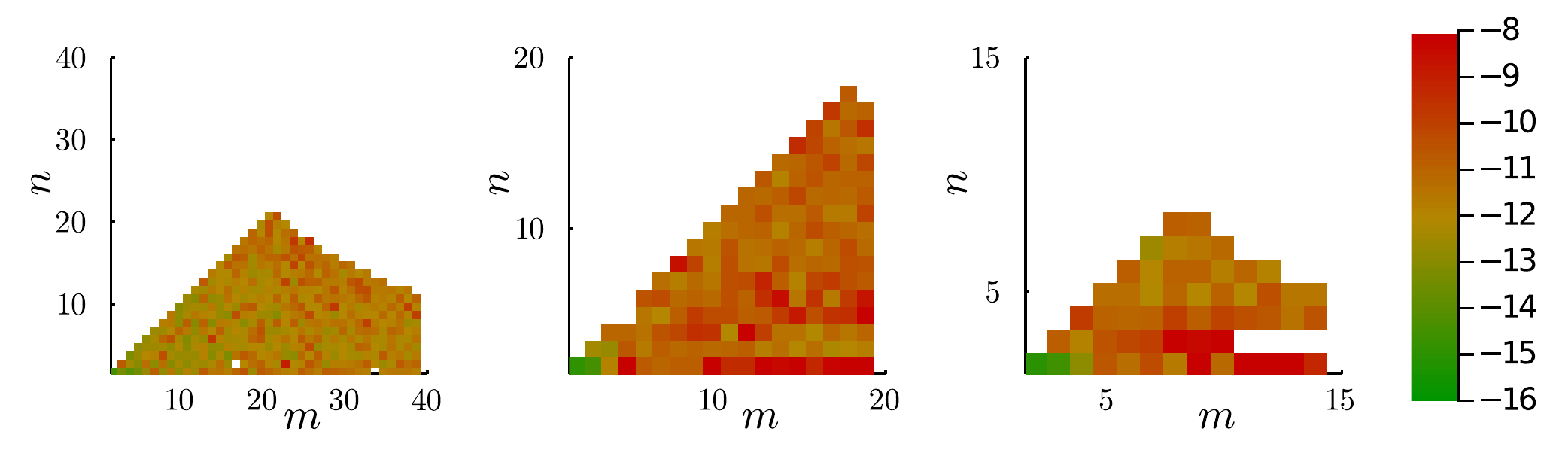}\\
\centerline{\small (b) \texttt{cpd\_ddl}}
 
\caption{A comparison of the $\log_{10}$ of the relative backward error of the proposed \algoname{} and the state-of-the-art \texttt{cpd\_ddl} from \cite{domanov2014canonical,domanov2017canonical} on random rank-$r$ tensors in  $\R^{\ell+1} \otimes \R^{m+1} \otimes \R^{n+1}$ with $\ell\ge m\ge n$. The largest dimension and the rank satisfy $\ell+1 = r = \min\{\mathcal{R}(m,n,(d,1)), mn\}$, where the function $\mathcal{R}$ is as in \cref{eqn_Rc_expression}. The outcomes for $d = 2, 3, 4$ are shown respectively in the left, middle, and right plots. The color scale is the same in all plots.} 
\label{fig_accuracy}
\end{figure}

For the first configuration \texttt{cpd\_ddl} was executed on \textit{bigmem} nodes of the supercomputer. These are the same as the \textit{skylake} nodes, except that they are equipped with $768$GB of main memory. The reason was that our Matlab driver routine consumed more than $162$GB of memory while running through all configurations.

The accuracy of \algoname{} and \texttt{cpd\_ddl} in this set of experiments is shown in \cref{fig_accuracy}. The newly proposed algorithm is consistently several orders of magnitude more accurate in terms of the relative backward error than the state of the art.

It took over 146 hours to obtain the (incomplete) results in the left panel of \cref{fig_accuracy}(b). To put this in perspective, we note that the computation for the left panel of \cref{fig_accuracy}(a) took a little over 17 hours. The missing values inside of the triangles in \cref{fig_accuracy}(b) indicate that \texttt{cpd\_ddl} wanted to allocate an auxiliary matrix that would require more than $16.2$GB of memory. The same memory constraint was also imposed on \algoname{}, but here only the largest case with $d=4$ and $m=n=14$ could not be treated.
The red pixel for $d = 2, m = 35, n = 13$ in \cref{fig_accuracy}(a) indicates that \algoname{} gave inaccurate results. This is the only case where \texttt{eigs} failed to find a sufficiently accurate nullspace. 

\begin{figure}[h!]
\includegraphics[width=\textwidth]{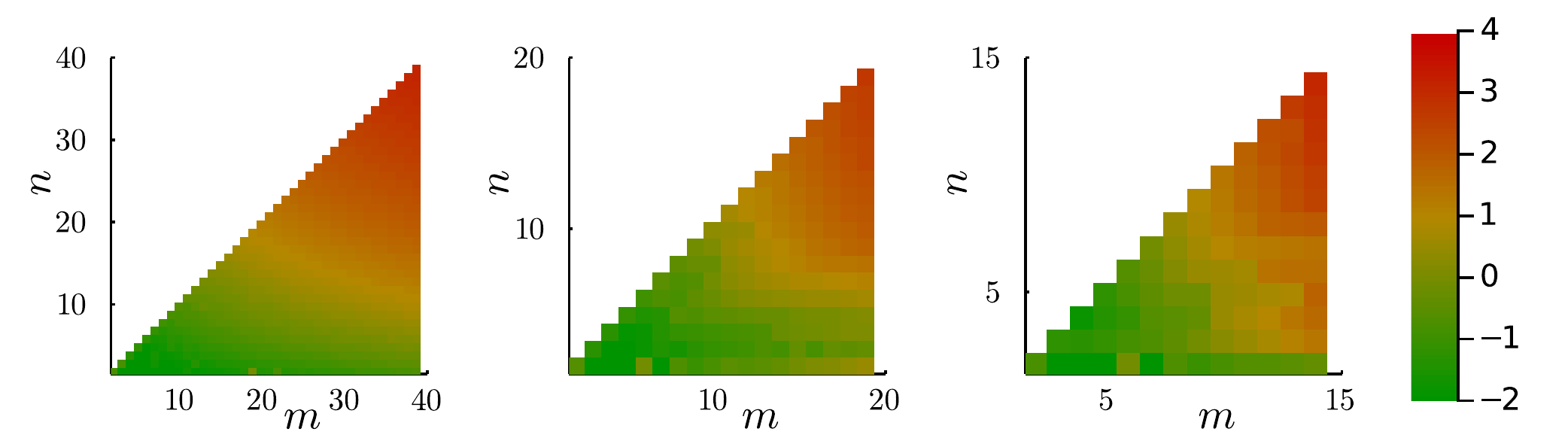}
\caption{The $\log_{10}$ of the total execution time (seconds) of \algoname{} in the setup from \cref{fig_accuracy}.}
\label{fig_timings}
\end{figure} 

\begin{figure}[h!]
\includegraphics[width=\textwidth]{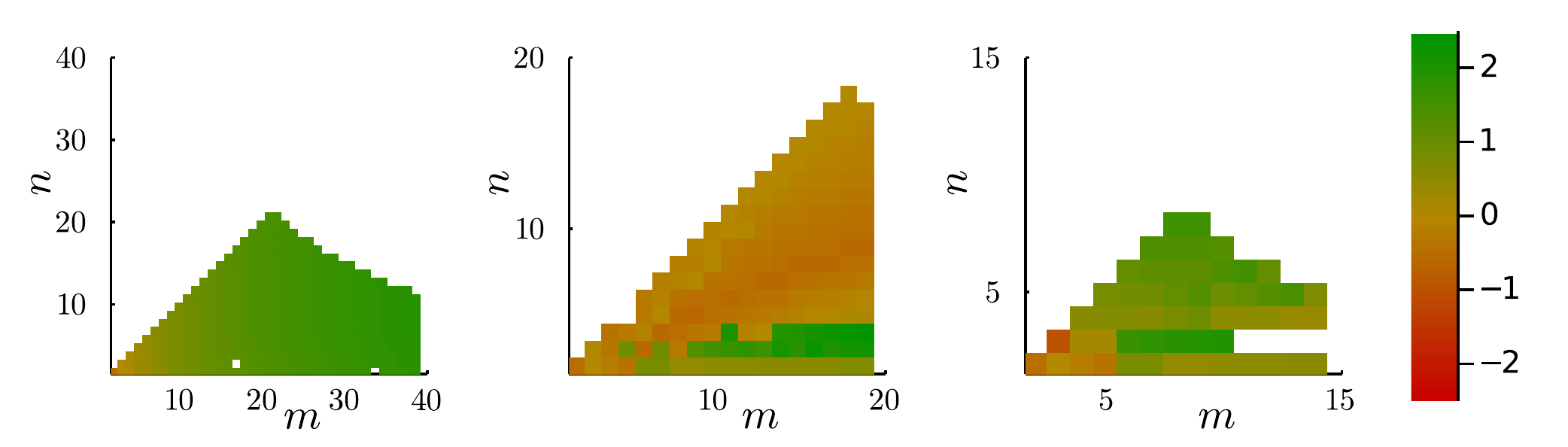}
\caption{The $\log_{10}$ of the size factor $\mu$ of \texttt{cpd\_ddl} relative to \algoname{}. The value $a$ means that \texttt{cpd\_ddl} consumes $10^a\times$ the memory \algoname{} needs to store $R_I(d,e)$.} 
\label{fig_scale_factor}
\end{figure}

The timings of our Julia implementation are shown in \cref{fig_timings}. 
As \texttt{cpd\_ddl} is implemented in a different programming language, we believe a direct comparison in timings is not opportune. Nevertheless, in both algorithms computing the kernel of the auxiliary matrix has the highest asymptotic time complexity. In \texttt{cpd\_ddl} it is an ${\rm L} \times {\rm L} $ square matrix, and in our algorithm, depending on the choice of degree, $R_{I}(d,e)$ is an almost-square ${\rm M} \times {\rm N}$ matrix with ${\rm M} \approx {\rm N}$. Therefore, we decided to plot the ratio between the number of elements in the auxiliary matrix of \texttt{cpd\_ddl} and \algoname{}. \Cref{fig_scale_factor} visualizes this factor $\mu = {\rm L}^2/({\rm M}{\rm N})$ in a logarithmic scale. This number indicates the fraction of memory that \texttt{cpd\_ddl} requires relative to \algoname{}. Raising it to the power $\frac{3}{2}$ gives an estimate of the speedup factor in execution time of \algoname{} relative to \texttt{cpd\_ddl}. 
Based on this estimate and the execution times we logged, it is accurate to state that the newly proposed algorithm outperforms the state of the art by up to two orders of magnitude for larger tensors.

\subsection{Robustness in the noisy case} \label{sec_sub_noisy}
The next experiment illustrates that \algoname{} can successfully decompose tensors even in the presence of some noise. The setup is as follows. We generate a random rank-$r$ tensor $\A$ of size $150 \times 25 \times 10$ by randomly sampling the $\alpha_i$,$\beta_i$, and $\gamma_i$ as before. Then, we add white Gaussian noise of relative magnitude $10^{e}$ for $e=-1,\ldots,-15$; that is, we compute $\A' = \A + 10^{e} \frac{\|\A\|_F}{\|\mathcal{E}\|_F} \mathcal{E}$. We provide $\A'$ as input to our algorithm and request a rank-$r$ decomposition. 

\begin{figure}[t]
\includegraphics[width=\textwidth]{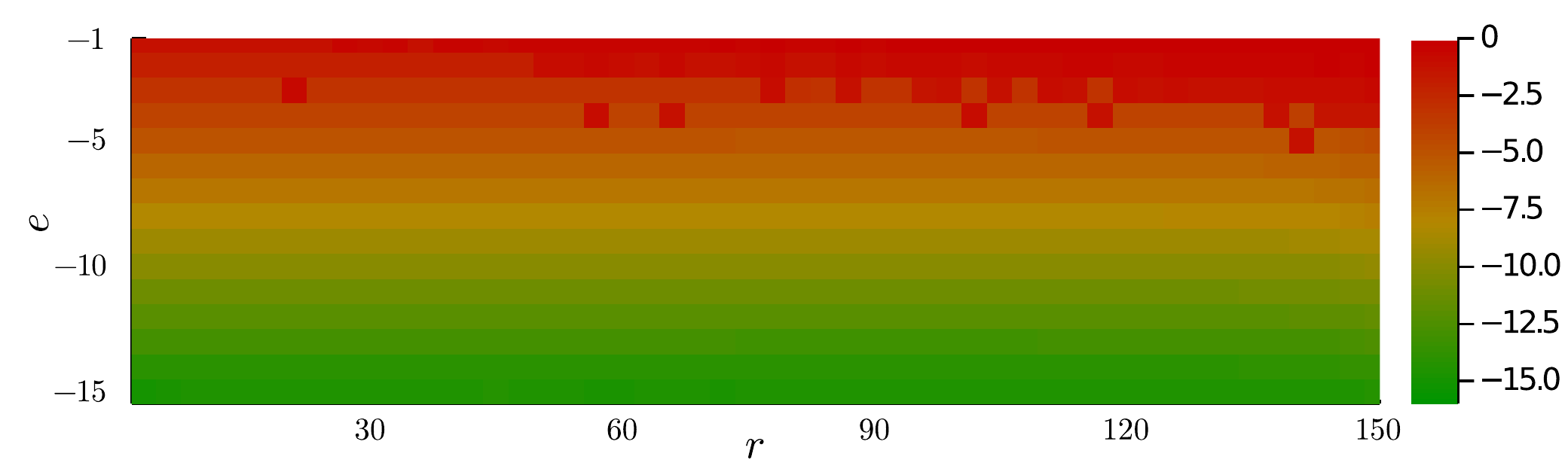}
\caption{The $\log_{10}$ of the relative backward error of decomposing a random rank-$r$ tensor of size $150 \times 25 \times 10$, corrupted by additive white Gaussian noise of relative magnitude $10^{e}$. The setup is described in detail in \cref{sec_sub_noisy}.} 
\label{fig_noise}
\end{figure}

The relative backward error between $\A'$ and the computed rank-$r$ decomposition is shown in a logarithmic scale in \cref{fig_noise}. Because of our setup, the rank-$r$ CPD of $\A$ has relative backward error $10^{e}$. A good tensor decomposition algorithm should thus return a CPD with a backward error of at most $10^{e}$. Remarkably, for tensors with random rank-$r$ CPDs, the proposed algorithm consistently manages to reach this benchmark when $e \le -5$. For ranks up to about half the maximum range (from $r=1$ to $70$), it even consistently manages to reach this benchmark for white noise of magnitude at most $10^{-2}$! Based on these results, we anticipate that \algoname{} could be employed as a rank-$r$ \emph{approximation} algorithm in the high signal-to-noise regime. We believe this observation warrants further research.

\subsection{An example of higher-order tensors} \label{subsec:higherorder}
The last experiment illustrates that the reshaping trick combined with a decomposition algorithm that works well in the unbalanced regime (such as \algoname{} and \texttt{cpd\_ddl} \cite{domanov2014canonical,domanov2017canonical}) is a powerful technique for decomposing high-order, high-rank tensors, even with a balanced shape.

As an example, we generated an eighth-order real tensor of size $7 \times 7 \times 7 \times 7 \times 6 \times 6 \times 5 \times 5$ and rank $r = 1000$ with factor matrices whose entries are sampled i.i.d.~from a standard normal distribution. Here is what happens:
\small
\begin{verbatim}
Grouping [[4, 6, 7, 8], [1, 2, 3], [5]] and reshaped to 
(1050, 343, 6) tensor in 0.810859903 s
1. Performed ST-HOSVD compression to (1000, 343, 6) in 0.44888179 s
    Swapped factors 2 and 3, so the tensor has size (1000, 6, 343)
    Selected degree increment d_0 = [1, 0]
2. Constructed kernel of A_1 of size (1000, 2058) in 22.333814068 s
3. Constructed resultant map of size (7203, 6348) in 72.176802896 s
4. Constructed res res' in 1.266772414 s
5. Computed cokernel of size (1000, 7203) in 108.332858902 s
6. Constructed multiplication matrices in 2.037837294 s
7. Diagonalized multiplication matrices and extracted solution 
    in 78.097176017 s
8. Refined factor matrices Y and Z in 151.170096114 s
9. Recovered factor matrix X in 0.186951757 s
10. Recovered the full factor matrices in 0.9396202759999999 s
Computed tensor rank decomposition in 440.457582263 s
Relative backward error = 3.873171296624731e-15
Relative forward error = 7.303893102189592e-14
\end{verbatim}
\normalsize

To our knowledge, this computation represents the first time any tensor decomposition algorithm of any type (i.e., alternating least squares, optimization-based methods, direct algebraic methods, homotopy-based methods, or heuristic methods) successfully decomposes a high-order rank-$1000$ tensor that cannot be reshaped to an order-$3$ tensor whose CPD can be computed with a pencil-based algorithm.

\section{Conclusions} \label{sec:conclusions}
The \algoname{} algorithm proposed in this paper computes the CPD of tensors satisfying \cref{assum:vanishingideal} using numerical linear algebra techniques for solving systems of polynomial equations. Its complexity is governed (\cref{prop_exponential,prop_polynomial}) by the regularity of a homogeneous, $\N^2$-graded ideal obtained from a flattening. We presented a formula for degrees $(d,1)$ in the regularity for generic tensors of many formats (\cref{thm:conjcases}) and proposed \cref{conj:reg} for the general case. Our experiments show that \algoname{} produces backward errors that are almost always of the order of the machine precision. This improves upon the previous state of the art by several orders of magnitude (see \cref{fig_accuracy}). In the high signal-to-noise-ratio regime, it seems the algorithm can be used to approximate noisy rank-$r$ tensors. 

Possible directions for future research include a further analysis of the regularity, a relaxation of the conditions in \cref{assum:vanishingideal}, generalizations for (semi-)symmetric tensors, and a theoretical analysis of \algoname{} in the noisy setting. 

\section*{Acknowledgements}
We are grateful to Alessandra Bernardi, Fulvio Gesmundo and Bernard Mourrain for fruitful discussions. We thank Ignat Domanov and Lieven de Lathauwer for kindly sharing a Matlab implementation of \texttt{cpd\_ddl} prior to its publication in the Tensorlab+ repository.

We thank the three anonymous reviewers for their detailed feedback on the manuscript that in particular led to the inclusion of \cref{subsec:reshapingtrick,sec_sub_accuracy}. In particular, we thank an anonymous referee for suggesting shorter proofs of \cref{prop:HFusc} and \cref{thm:boundr}.

Some of the resources and services used in this work were provided by the VSC (Flemish Supercomputer Center), funded by the Research Foundation---Flanders (FWO) and the Flemish Government.

Nick Vannieuwenhoven was partially supported by the Postdoctoral Fellowship with number 12E8119N from the Research Foundation---Flanders (FWO).


\appendix
\section{Proof of the technical results} \label{app:proofs}
For two ideals $I, K \subset S$, we write
$ (I: K^\infty) = \{ f \in S ~|~ K^k f \subset I \text{ for some } k \in \N \}.$ The next result follows from \cref{lem:zerodim}.

\begin{corollary} \label{cor:satisvan}
Under \cref{assum:vanishingideal}, we have the identity $(I: K^\infty) = J$, where $I = \ideal{\ker \A_{(1)}}$, $J$ is the vanishing ideal 
\[ 
J = \ideal{ f \in S ~|~ f \text{ is homogeneous and } f(\z_i) = 0, i = 1, \ldots, r },
\]
and $K = \ideal{x_iy_j ~|~ 0 \leq i \leq m, 0 \leq j \leq n}$ is the irrelevant ideal of $S$.
\end{corollary}
\begin{proof}
By \cref{lem:zerodim}, the modules $S/I$ and $S/J$ define the same subscheme of $X$ consisting of $r$ simple points. It follows that $K^k (J/I) = 0$ for large enough $k$. See for instance \cite[Proposition 5.3.10]{cox2011toric}.
\end{proof}

\begin{proof}[Proof of \cref{lem:h0}]
This follows from the fact that all $h \in S_{(d',e')}$ satisfying $h(\z_i) = 0$ lie on a hyperplane $H_i$ through the origin in the $\C$-vector space $S_{(d',e')}$. Any $h_0$ in the complement of $H_1 \cup \cdots  \cup H_r$ satisfies $h_0(\z_i) \neq 0, i = 1, \ldots, r$.
\end{proof}
\begin{proof}[Proof of \cref{lem:hf}]
Let $J = (I: K^\infty)$ be the saturation of $I$ with respect to the irrelevant ideal $K$ (see \Cref{cor:satisvan}). First, it follows from \cite[Propositions 4.4 and 6.7]{maclagan2004multigraded} that there exists $(\delta_1,\epsilon_1) \in \N^2$ such that $\HF_{S/J}(\delta_1+\delta',\epsilon_1+\epsilon') = r$ for all $(\delta',\epsilon') \in \N^2$. Secondly, we show that there exists $(\delta_2,\epsilon_2) \in \N^2$ such that $I_{(\delta_2 + \delta',\epsilon_2 + \epsilon')} = J_{(\delta_2 + \delta',\epsilon_2 + \epsilon')}$ for all $(\delta',\epsilon') \in \N^s$. To see this, note that $K^k J \subset I$ for some $k \in \N$ (see e.g.\ \cite[Chapter 4, Section 4, Proposition 9]{cox2013ideals}) and $K^k = \ideal{S_{(k,k)}}$. For a finite set of homogeneous generators $g_1, \ldots, g_m$ of $J$, choose $(\delta_2,\epsilon_2)$ such that entry-wise, $(\delta_2,\epsilon_2) \geq (k, k) + \deg(g_i),~ i = 1, \ldots, m$. Take $(d,e) \in  (\max(\delta_1,\delta_2), \max(\epsilon_1,\epsilon_2)) + \N^2$ such that $(d,e) \neq (1,1)$ and $(d-1,e-1) \geq (0,0)$.
\end{proof}

\begin{proof}[Proof of \cref{thm:MHNFlowrank}]
Under the assumptions of the theorem, $((0,1),(1,0))$ is a regularity pair for the ideal $J = (I:K^\infty)$, defining the same set of points $V_X(J) = V_X(I)$ (see \cref{cor:satisvan}). The statement now follows from \cite[Theorem 5.5.3, Propositions 5.5.4 and 5.5.5]{telen2020thesis}.
\end{proof}

\begin{proof}[Proof of item (3) in \cref{thm:conjcases}]
We prove the case $m = 2$ in item (3) of the theorem. It suffices to show that for some configuration $Z$, the matrix
\[ 
A' = \begin{bmatrix}
-\Gamma H_0 & \\
 & - \Gamma H_0 \\
 \Gamma H_2 & - \Gamma H_1
\end{bmatrix}
\] 
has full rank. We set $\beta_{i0} = 1, i = 1, \ldots, r$, such that $H_0$ is the identity matrix. The condition \eqref{eq:goodregion} ensures that $A'$ has more rows than columns, so it suffices to show that $\ker A' = \{0\}$. Suppose $A' v = 0$, then $v$ can be split into $v_1, v_2 \in \C^r$ such that $\Gamma v_1 = \Gamma v_2 = 0$. If $r \leq n+1$, it is clear that this implies $v_1 = v_2 = 0$ for generic $Z$. Therefore, we assume $r > n + 1$ and make the following choice for $\Gamma$: 
\[
\Gamma = \begin{bmatrix}
1 &  & & & \gamma_{n+2,0} & \cdots & \gamma_{r0} \\
  & 1& & & \gamma_{n+2,1} & \cdots & \gamma_{r1} \\
  &  & \ddots & & \vdots & & \vdots \\ 
  &  &  & 1 & \gamma_{n+2,n} & \cdots & \gamma_{rn}  
\end{bmatrix}  = \begin{bmatrix} \id_{n+1} & \hat{\Gamma} \end{bmatrix} \in \C^{(n+1) \times r},
\]
where $\hat{\Gamma} \in \C^{(n+1) \times \kappa}$ is the submatrix of $\Gamma$ consisting of its last $\kappa = r - (n+1)$ columns. We have that $\Gamma v_i = 0$ implies $v_i = \Gamma^\perp w_i$ for some $w_i \in \C^\kappa$ and $\Gamma^\perp = \left [ \begin{smallmatrix}
-\hat{\Gamma} \\ \id_\kappa
\end{smallmatrix} \right ]$. Hence $Av = 0$ is equivalent to 
$ \begin{bmatrix}
\Gamma H_2 \Gamma^\perp & -\Gamma H_1 \Gamma^\perp
\end{bmatrix} \left [ \begin{smallmatrix}
w_1 \\ w_2
\end{smallmatrix} \right ] = 0.$
The condition \eqref{eq:goodregion} implies $2 \kappa \leq n+1$, so that the coefficient matrix in this equation has more rows than columns. Upon closer inspection, we see that 
\[ 
\Gamma H_q \Gamma^\perp 
= \begin{bmatrix}
(\beta_{n+2,q} - \beta_{1q})\gamma_{n+2,0} & \cdots & (\beta_{rq} - \beta_{1q}) \gamma_{r0} \\
(\beta_{n+2,q} - \beta_{2q})\gamma_{n+2,1} & \cdots & (\beta_{rq} - \beta_{2q}) \gamma_{r1} \\
\vdots & & \vdots \\
(\beta_{n+2,q} - \beta_{n+1,q})\gamma_{n+2,n} & \cdots & (\beta_{rq} - \beta_{n+1,q}) \gamma_{rn} \\
\end{bmatrix}
= \begin{bmatrix}
   (\beta_{jq} - \beta_{i+1,q})\gamma_{ji}
  \end{bmatrix}_{\substack{0\le i\le n,\\ n+2\le j\le r}}.
\]
In order to make rows $\kappa+1, \ldots, 2 \kappa$ equal to zero in $\Gamma H_2 \Gamma^\perp$, we set $\beta_{\kappa + 1, 2} = \beta_{\kappa+2, 2} = \cdots = \beta_{2 \kappa,2} = \beta_{n+2,2} = \beta_{n+3,2} = \cdots = \beta_{r,2} = 1$. All other $\beta$-coordinates are chosen at random, such that all entries of $\begin{bmatrix}
\Gamma H_2 \Gamma^\perp & -\Gamma H_1 \Gamma^\perp\end{bmatrix} $, except those in the rows $\kappa+1, \ldots, 2 \kappa$ of $\Gamma H_2 \Gamma^\perp$, are of the form $\star \gamma_{ji}$, with $\star$ some non-zero complex number. Then,
\[ 
\begin{bmatrix}
\Gamma H_2 \Gamma^\perp & -\Gamma H_1 \Gamma^\perp\end{bmatrix}  = \begin{bmatrix}
C_{11} & C_{12} \\ 0 & C_{22} \\
D_1 & D_2 
\end{bmatrix},
\]
where $C_{ij} \in \C^{\kappa \times \kappa}$ and $D_i \in \C^{(n+1 -2\kappa) \times \kappa }$. The minor corresponding to the first $2\kappa$ rows is seen to be $\det(C_{11}) \det(C_{22})$. This is a product of two nonzero polynomials in the parameters $\gamma_{ji}$, $i = 0, \ldots, n, j = n+2, \ldots, r$. For generic choices of the parameters, this minor is non-zero, hence $w_1 = w_2 = 0$, and thus $v = 0$ and $\ker A' = \{0\}$.
\end{proof}

 \bigskip
 \bigskip

 \noindent{\bf Authors' addresses:}
\medskip

\noindent Simon Telen, MPI-MiS Leipzig and CWI Amsterdam {\em (current)}
\hfill {\tt simon.telen@mis.mpg.de}

\noindent Nick Vannieuwenhoven, KU Leuven \hfill{\tt nick.vannieuwenhoven@kuleuven.be}
\end{document}